\documentclass[onefignum,onetabnum]{siamart190516}

\usepackage{lipsum}
\usepackage{amsfonts}
\usepackage{graphicx}
\usepackage{epstopdf}
\usepackage{algorithmic}
\usepackage{amssymb}
\ifpdf
  \DeclareGraphicsExtensions{.eps,.pdf,.png,.jpg}
\else
  \DeclareGraphicsExtensions{.eps}
\fi

\newsiamremark{remark}{Remark}
\newsiamremark{hypothesis}{Hypothesis}
\crefname{hypothesis}{Hypothesis}{Hypotheses}
\newsiamthm{claim}{Claim}

\headers{Staggered mesh schemes for dissipative systems}{Z. Liu,  N. Zheng and X. Li}

\title{Linear, decoupled and positivity-preserving staggered mesh schemes for general dissipative systems with arbitrary energy distributions\thanks{
Submitted to the editors DATE.
\funding{We would like to acknowledge the assistance of volunteers in putting together this example manuscript and supplement. The first author was supported by the Natural Science Foundation of Shandong Province (Grant Nos: ZR2023YQ007 and ZR2024MA077), the Taishan Scholars Program of Shandong Province of China (Grant No: tsqn202408140). The second author was partially supported by the Hong Kong Polytechnic University Postdoctoral Research Fund 1-W22P.  The third author was supported by the National Natural Science Foundation of China (Grant Nos: 12271302 and 12131014) and the Shandong Provincial Natural Science Foundation for Outstanding Youth Scholar (Grant No: ZR2024JQ030). 
}}}

\author{Zhengguang Liu\thanks{School of Mathematics and Statistics, Shandong Normal University, Jinan, Shandong, 250358, China. 
  (\email{liuzhg@sdnu.edu.cn}).}
 \and Nan Zheng\thanks{Corresponding author. Department of Applied Mathematics, The Hong Kong Polytechnic University, Hung Hom, Hong Kong. 
  	(\email{nanzheng@polyu.edu.hk}).}
  	  \and  Xiaoli Li\thanks{Corresponding author. School of Mathematics, Shandong University, Jinan, Shandong, 250100, China.
  		(\email{xiaolimath@sdu.edu.cn}).}
}

\usepackage{amsopn}

\usepackage[leqno]{amsmath}
\usepackage{graphicx}
\usepackage{graphics}
\usepackage{epstopdf}
\usepackage{subfig}
\usepackage{subfloat}
\usepackage{threeparttable}
\allowdisplaybreaks[4]
\ifpdf
\hypersetup{
  pdftitle={Linear, decoupled and positivity-preserving staggered mesh schemes for general dissipative systems with arbitrary energy distributions},
  pdfauthor={Zhengguang Liu, Nan Zheng, and Xiaoli Li}
}
\fi

\begin{document}
\maketitle

\begin{abstract}
In this paper, we develop a novel staggered mesh (SM) approach for general nonlinear dissipative systems with arbitrary energy distributions (including cases with known or unknown energy lower bounds). Based on this framework, we propose several second-order semi-discrete schemes that maintain linearity, computational decoupling, and unconditional energy stability. Firstly, for dissipative systems with known energy lower bounds, we introduce a positive auxiliary variable $V(t)$ to substitute the total energy functional, subsequently discretizing it on staggered temporal meshes to ensure that the energy remains non-increasing regardless of the size of time step. The newly developed schemes achieve full computational decoupling, maintaining essentially the same computational expense as conventional implicit-explicit methods while demonstrating significantly improved accuracy. Furthermore, we rigorously establish the positivity preservation of the discrete variable $V^{n+1/2}$ which is a crucial property ensuring numerical stability and accuracy. Theoretical analysis confirms second-order temporal convergence for the proposed SM schemes. Secondly, for dissipative systems lacking well-defined energy lower bounds, we devise an alternative auxiliary variable formulation and extend the SM framework to maintain unconditional energy stability while preserving numerical effectiveness and accuracy. Finally, comprehensive numerical experiments, including benchmark problem simulations, validate the proposed schemes' efficacy and demonstrate their superior performance characteristics.
\end{abstract}


\section{Introduction}
Dissipative systems are a class of important and broadly applicable dynamical systems that are characterized by the continuous flow of energy or matter into and out of the system, leading to a non-equilibrium steady state. They have a very wide range of applications across various scientific and engineering disciplines due to their ability to maintain order and complexity far from equilibrium. The class of dissipative systems described above includes not only gradient flows \cite{cheng2020global,huang2020highly,xiaoli2019energy,shen2019new} but also other dissipative systems that do not  have the gradient structure, such as Navier-Stokes equations \cite{diegel2017convergence,huang2022new,xiaoli2020error}, viscous Burgers equations \cite{beck2011using,gottlieb2012stability}, reaction-diffusion equations \cite{fu2023high,liu2021structure} etc.

Consider the following class of dissipative systems
\begin{equation}\label{dissipative-systems}
	\frac{\partial u}{\partial t}+\mathcal{A}u+g(u)=0,
\end{equation}
where $u$ is a scalar or vector function, $\mathcal{A}$ is a positive definite operator and $g(u)$ is a semi-linear or quasi-linear operator with lower-order derivatives. The above equation satisfies the following dissipative energy law:
\begin{equation}\label{dissipative-energy-law}
	\frac{dE_{tot}}{dt}=-\mathcal{K}(u),
\end{equation}
where $E_{tot}(u)$ is a free energy and $\mathcal{K}(u)\geq0$ for all $u$. 

It's worth noting that keeping the discrete energy dissipation law to eliminate non-physics numerical solutions is very crucial to simulate dissipative systems. In the last few decades of studies, a series of popular numerical schemes have been proposed to numerically solve various dissipative systems. The fully explicit and fully implicit schemes \cite{du1991numerical,elliott1987numerical} dominated early numerical simulation studies of dissipative systems. The fully explicit scheme is simple but has very strict requirements on the time step size to maintain the structure of the original systems. The implicit scheme can usually keep the physical structure but requires harsh time-step limitations to ensure the uniqueness of the numerical solution, and because of the implicit treatment of nonlinear terms, nonlinear iterations are required for the computation. Convex splitting approach is one of the well-known methods for constructing an unconditionally unique and energy stable numerical scheme, which was first proposed by Eyre \cite{eyre1998unconditionally} in the study of the Cahn-Hilliard equation. Later, this method is extended and successfully applied to other dissipative systems, especially to various phase field models \cite{chen2012linear,furihata2001stable,gomez2011provably,guan2014convergent}. However, the convex splitting methods are not easy to simulate complex dissipative systems and its numerical schemes are nonlinear schemes in most cases which need to be solved iteratively. The stabilized schemes have been proved to be linear, energy stable schemes and still effective at large step sizes.  Chen and Shen \cite{chen1998applications} firstly gave a stabilized semi-implicit scheme based on the Fourier spectral method by treating the linear terms implicitly and the nonlinear terms explicitly for the Cahn-Hilliard equation. Since then, many stabilized schemes \cite{li2017second,li2017stabilization,shen2010numerical,tang2020efficient} have been developed and widely used in various dissipative systems. Xu et al. \cite{xu2019stability} found that the stabilized semi-implicit scheme provides an approximation to the original solution of the Allen-Cahn model at a delayed time. Du et al. \cite{du2021maximum} proposed an exponential time differencing (ETD) approach for a class of semilinear parabolic equations to develop first- and second-order accurate temporal discretization schemes, that satisfy the maximum bound principle unconditionally in the time-discrete setting. Besides, Yang et al. \cite{yang2018linear,yang2017numerical} gave a class of numerical schemes based on invariant energy quadraticization methods and the basic idea is to rewrite the energy functional into a quadratic form by introducing an auxiliary variable to obtain a new equivalent system. Recently, the SAV approach introduced by Shen et al. \cite{shen2018scalar,shen2019new} has been attracted much attention in numerical solutions for gradient flows due to its inherent advantage of preserving energy dissipation law. In \cite{huang2022new}, a high-order generalized SAV (GSAV) approach was proposed for general dissipative systems. In addition, some other numerical methods can also be used to solve dissipative systems efficiently. We refer to \cite{cheng2020new,cheng2015fast,guo2016local,li2022new,qian2020gpav,qiao2015stability} and references therein for a more complete literature on this subject.

The SAV and its extended schemes \cite{akrivis2019energy,cheng2020new,hou2019variant} mostly require twice as much computation as the semi-implicit schemes. Furthermore, although the introduced auxiliary variables are positive, most of them can not maintain the positive property for the discrete auxiliary variables, which is very important for long time simulation and correct evolution of phase transitions. The GSAV schemes \cite{huang2022new,huang2020highly,zhang2022generalized} only require solving one linear system with constant coefficients at each time step and their auxiliary variable is always positive. However, the convergence order of the discrete energy is only first order which will affect the dynamic evolution of energy dissipation. In this paper, we consider a novel linear and second-order staggered mesh (SM) approach for dissipative systems and carry out a rigorous error analysis. The new proposed SM schemes enjoy several distinct advantages, including:
\begin{itemize}
	\item the discrete variables can be fully decoupled in terms of the calculation, so its computational cost is essentially the same as the usual implicit-explicit schemes, but the accuracy is much higher than the implicit-explicit schemes;
	\item the discrete energy is unconditionally stable and achieves second-order convergence rate, which is higher than the classical GSAV scheme with only first-order accuracy for the modified energy;
	\item the extended scheme has been proposed and verified to be effective for simulating dissipative systems with unknown energy lower bound;
	\item rigorous error estimates with a general framework can be established for the constructed SM schemes. 
\end{itemize}

The rest of this paper is organized as follows. In Section 2, we present the SM approach for dissipative systems. We first give a general SM scheme to discretize dissipative systems with known energy lower bound. Secondly, an extended SM scheme is considered to simulate dissipative systems with unknown energy lower bound. In Section 3, we carry out an error analysis for the proposed SM scheme. We present a comparison of SM approach with GSAV and Lagrange multiplier approaches in Section 4, and provide ample numerical examples including some benchmark problems to validate its effectiveness.

\section{Staggered mesh approach}\label{SMmethod}
In this section, we will develop a novel staggered mesh approach and construct several linear and efficient schemes based on the staggered meshes in time for general dissipative systems with known and unknown energy lower bounds.

\subsection{Staggered mesh scheme for dissipative systems with known energy lower bound}    
The majority of dissipative systems have a known energy lower bound, e.g., there exists a constant $C_0>0$ such that $E(u)=E_{tot}(u)+C_0>0$. Introduce 
an auxiliary variable $V(t)=E(u)$ to satisfy $V(t)>0$. We first take the logarithm of the auxiliary variable  $V(t)$ and then take the time derivative of $\ln V(t)$ to obtain
\begin{equation}\label{Modified-dissipative-energy-law2}
	\frac{d\ln V}{dt}=\frac1V\frac{dV}{dt}=-\frac{\mathcal{K}(u)}{E(u)}.
\end{equation}  

Introduce another auxiliary variable $\eta(t)$ which satisfies $\eta(t)\equiv1$ at the continuous level and combine the above equation \eqref{Modified-dissipative-energy-law2} with the dissipative systems \eqref{dissipative-systems} to obtain the following equivalent system:
\begin{equation}\label{equivalent-dissipative-systems}
	\left\{
	\begin{array}{ll}
		\displaystyle\frac{\partial u}{\partial t}+\mathcal{A}u+g\left(\eta(t) u\right)=0,\\
		\displaystyle\frac{d\ln V}{dt}=-\frac{\mathcal{K}(u)}{E(u)},\\
		\displaystyle \eta(t)=\chi(V).
	\end{array}
	\right.
\end{equation}
Here $\chi(V)$ is a functional of $V$ which is equal to 1 at the continuous level and the selection of its specific forms is rather general.

We notice that the second equation in the energy expression \eqref{Modified-dissipative-energy-law2} is actually equivalent to $\frac{dV}{dt}=-\frac{V(t)}{E(u)}\mathcal{K}(u)$ which is essentially consistent with the equality related to auxiliary variable in the GSAV method in \cite{huang2022new}. However, we replace the equation of auxiliary variables in GSAV by the above second expression in \eqref{equivalent-dissipative-systems}, which can realize that the term on the right-hand side of the equation does not have any related forms of $V(t)$, so that it is beneficial for us to discretize the unknowns $u$ and $V$ on the staggered mesh.

From equation \eqref{Modified-dissipative-energy-law2}, one can see that the above system satisfies a modified energy dissipation law:
\begin{equation*}
	\frac{dV}{dt}=V\frac{d\ln V}{dt}=-\frac{V(t)}{E(u)}\mathcal{K}(u)=-\mathcal{K}(u)\leq0.
\end{equation*}

Before giving a semi-discrete formulation, we let $N>0$ be a positive integer and set
\begin{equation*}
	\Delta t=T/N,\quad t^n=n\Delta t,\quad \text{for}\quad n\leq N.
\end{equation*}

Considering the Crank-Nicolson semi-implicit scheme for the first equation in \eqref{equivalent-dissipative-systems} and discretizing the second equation in \eqref{equivalent-dissipative-systems} on the staggered time mesh, we can obtain the following CN-SM scheme:
\begin{equation}\label{CN-SM1}
	\begin{array}{l}
		\displaystyle\frac{u^{n+1}-u^{n}}{\Delta t}+\mathcal{A}\frac{u^{n+1}+u^{n}}{2}+ g\left(\eta^{n+\frac12}\overline{u}^{n+\frac12}\right)=0,\\
		\displaystyle\frac{\ln V^{n+
				\frac12}-\ln V^{n-\frac12}}{\Delta t}=-\frac{\mathcal{K}(u^n)}{E(u^n)},\\
		\displaystyle \eta^{n+\frac12}=\chi(V^{n+\frac12}),
	\end{array}
\end{equation}
where $\overline{u}^{n+\frac12}$ is any explicit $O(\Delta t^2)$ approximation for $u(t^{n+\frac12})$. For example, we can choose $\overline{u}^{n+\frac12}=\frac32u^n-\frac12u^{n-1}$ for $n\geq1$ or we can use a simple first-order scheme to obtain it, such as the semi-implicit scheme
$$\displaystyle\frac{\overline{u}^{n+\frac12}-u^n}{\Delta t/2}+\mathcal{A}\overline{u}^{n+\frac12}+g(u^n)=0,$$ which has a local truncation error of $O\left(\Delta t^2\right)$.

As for the stability, we have the following result.
\begin{theorem}\label{theorem2}
	Given $V^{n-\frac12}>0$, we have $V^{n+\frac12}>0$, and the CN-SM scheme \eqref{CN-SM1} is unconditionally energy stable in the sense that
	\begin{equation}\label{CN-SM2}
		V^{n+\frac12}-V^{n-\frac12}=(e^{-\frac{\mathcal{K}(u^n)}{E(u^n)}\Delta t}-1)V^{n-\frac12}\leq0.
	\end{equation}
\end{theorem}
\begin{proof}
	Noting that $V^{n-\frac12}>0$ and using properties of logarithmic function for the second equation in the CN-SM scheme \eqref{CN-SM1}, we can get
	\begin{equation*}
		\displaystyle\ln\frac{V^{n+\frac12}}{V^{n-\frac12}}=-\frac{\mathcal{K}(u^n)}{E(u^n)}\Delta t.
	\end{equation*}
	
	Taking the exponential function on both sides of the above equation yields
	\begin{equation*}
		\displaystyle V^{n+\frac12}=e^{-\frac{\mathcal{K}(u^n)}{E(u^n)}\Delta t}V^{n-\frac12}.
	\end{equation*}
	
	Since $\mathcal{K}(u^n)\geq0$ and $E(u^n)>0$, we have
	\begin{equation*}
		0<e^{-\frac{\mathcal{K}(u^n)}{E(u^n)}\Delta t}\leq1.
	\end{equation*}
	
	We then immediately conclude 
	\begin{equation*}
		0<V^{n+\frac12}\leq V^{n-\frac12},
	\end{equation*}
	and
	\begin{equation*}
		V^{n+\frac12}-V^{n-\frac12}=(e^{-\frac{\mathcal{K}(u^n)}{E(u^n)}\Delta t}-1)V^{n-\frac12}\leq0,
	\end{equation*}
	which completes the proof.
\end{proof}
\begin{remark} \label{rem_choose chi}
	The choice for $\chi$ is rather general. Some examples are:
	\begin{align*}
		\chi_1(V)&=\frac{V(t)}{E(u)},\quad \chi_2(V)=x^{V(t)-E(u)},\\ \chi_3(V)&=\cos\left(V(t)-E(u)\right), \quad \chi_4(V)=1+V(t)-E(u).
	\end{align*} 
	The corresponding discrete schemes will be
	\begin{equation*}
		\aligned
		&\chi_1(V^{n+\frac12})=\frac{V^{n+\frac12}}{E(\overline{u}^{n+\frac12})},\quad\chi_2(V^{n+\frac12})=x^{V^{n+\frac12}-E(\overline{u}^{n+\frac12})},\\
		&\chi_3(V^{n+\frac12})={\cos}\left(V^{n+\frac12}-E(\overline{u}^{n+\frac12})\right), \quad \chi_4(V)=1+V^{n+\frac12}-E(\overline{u}^{n+\frac12}).
		\endaligned
	\end{equation*} 
\end{remark}

Next, we will show that the above scheme \eqref{CN-SM1} can be efficiently implemented. We first use a simple first-order scheme to obtain $\overline{u}^{\frac12}$, such as the following semi-implicit scheme
\begin{equation*}
	\displaystyle\frac{\overline{u}^{\frac12}-u^0}{\Delta t/2}+\mathcal{A}\overline{u}^{\frac12}+g(u^0)=0,
\end{equation*}
which has a local truncation error of $O(\Delta t^2)$.

Secondly, giving $V^{\frac12}=E(\overline{u}^{\frac12})$, we then calculate $u^1$ by the following classical semi-implicit scheme:
\begin{equation*}
	\displaystyle\frac{u^{1}-u^0}{\Delta t}+\mathcal{A}\frac{u^{1}+u^{0}}{2}+g(\overline{u}^{\frac12})=0.
\end{equation*}

Setting $n=1$, we can then obtain $V^{n+\frac12}$ by the following equation:
\begin{equation}\label{SM-V}
	\displaystyle V^{n+\frac12}=e^{-\frac{\mathcal{K}(u^n)}{E(u^n)}\Delta t}V^{n-\frac12}.
\end{equation}

Once $ V^{n+\frac12}$ is known, we can calculate $\eta^{n+\frac12}$ by the following
$$\eta^{n+\frac12}=\chi(V^{n+\frac12}).$$

Letting $\overline{u}^{n+\frac12}=\frac12(3u^n-u^{n-1})$, we finally calculate $u^{n+1}$ as follows:
\begin{equation*}
	\displaystyle\frac{u^{n+1}-u^{n}}{\Delta t}+\mathcal{A}\frac{u^{n+1}+u^{n}}{2}+g\left(\eta^{n+\frac12}\overline{u}^{n+\frac12}\right)=0.
\end{equation*}

Setting $n=n+1$ and going back to equation \eqref{SM-V}, we can calculate $V^{n+\frac12}$, $\eta^{n+\frac12}$ and $u^{n+1}$ step-by-step in order. 

\begin{remark}
	The discrete variables $V^{n+\frac12}$ and $u^{n+1}$ are calculated decoupled in the CN-SM scheme \eqref{CN-SM1}. It means the above procedure only requires solving one linear equation with constant coefficients as in a standard semi-implicit scheme.
\end{remark}

\begin{remark}
	In order to prevent the uncontrollable factors brought by the exponential function calculation in \eqref{SM-V}, we can redefine $V(t)=\theta E(u)+C>0$ and replace the second equation in \eqref{equivalent-dissipative-systems} as the following:
	\begin{equation*}
		\frac{d\ln V}{dt}=\frac{\theta}{\theta E(u)+C}\frac{dE}{dt}=-\frac{\theta\mathcal{K}(u)}{\theta E(u)+C},
	\end{equation*} 
	where $\theta>0$ is a tunable parameter. When $\mathcal{K}(u)\gg0$, we can choose a relative small $\theta$ to control $\theta\mathcal{K}(u)$. 
\end{remark}

\begin{remark}
	The proposed staggered mesh method is also efficient for the other positive definition of the introduced auxiliary variable $V(t)$. For example, we can choose $V(t)=\sqrt{E(u)}$ or $V(t)=\exp(E_{tot}(u))$ to replace $V(t)=E(u)$. Consequently, the second equation in the equivalent systems \eqref{equivalent-dissipative-systems} will be replaced by the following new equation 
	\begin{equation*}
		\frac{d\ln V}{dt}=-\frac{\mathcal{K}(u)}{2E(u)},
	\end{equation*} 
	or
	\begin{equation*}
		\frac{d\ln V}{dt}=-\mathcal{K}(u).
	\end{equation*} 
\end{remark}
\begin{remark}
	In the proposed CN-SM scheme \eqref{CN-SM1}, we add a control factor $\eta^{n+\frac12}$ to modify $\overline{u}^{n+\frac12}$ in the nonlinear term $g(u)$. Another common way is to add the same control factor to modify the whole nonlinear term $g(\overline{u}^{n+\frac12})$, i.e.
	\begin{equation*}
		\displaystyle\frac{u^{n+1}-u^{n}}{\Delta t}+\mathcal{A}\frac{u^{n+1}+u^{n}}{2}+\eta^{n+\frac12}g\left(\overline{u}^{n+\frac12}\right)=0.
	\end{equation*}
\end{remark}
\begin{remark}
	We can also add a novel stabilized term for the first equation in \eqref{CN-SM1} to improve its stability:
	\begin{equation*}
		\displaystyle\frac{u^{n+1}-u^{n}}{\Delta t}+\mathcal{A}\frac{u^{n+1}+u^{n}}{2}+S\left(\frac{u^{n+1}+u^{n}}{2}-\eta^{n+\frac12}\overline{u}^{n+\frac12}\right)+g\left(\eta^{n+\frac12}\overline{u}^{n+\frac12}\right)=0.
	\end{equation*} 
	
	This will not compromise the energy dissipation property of the scheme.
\end{remark}
\begin{lemma}
	We can also obtain the following new SM scheme by swapping the discrete meshes of $u$ and $V$:
	\begin{equation*}
		\begin{array}{l}
			\displaystyle\frac{u^{n+\frac32}-u^{n+\frac12}}{\Delta t}+\mathcal{A}\frac{u^{n+\frac32}+u^{n+\frac12}}{2}+ g\left(\eta^{n+1}\overline{u}^{n+1}\right)=0,\\
			\displaystyle\frac{\ln V^{n+1}-\ln V^{n}}{\Delta t}=-\frac{\mathcal{K}(u^{n+\frac12})}{E(u^{n+\frac12})},\\
			\displaystyle \eta^{n+1}=\chi(V^{n+1}),
		\end{array}
	\end{equation*}
	where $\overline{u}^{n+1}$ is any explicit  approximation with $O(\Delta t^2)$ for $u(t^{n+1})$. It's not difficult to obtain the following modified energy dissipation law:
	\begin{equation*}
		V^{n+1}-V^{n}=(e^{-\frac{\mathcal{K}(u^{n+\frac12})}{E(u^{n+\frac12})}\Delta t}-1)V^{n}\leq0.
	\end{equation*} 
\end{lemma}
\begin{lemma}
	Inspired by the GSAV schemes in \cite{huang2020highly}, we can also construct $k$th order IMEX schemes for the expanded system \eqref{equivalent-dissipative-systems} as follows
	\begin{equation}\label{HSM-e1}
		\begin{array}{l}
			\displaystyle\frac{\alpha_k\overline{u}^{n+1}-\beta_k(u^n)}{\Delta t}+\mathcal{A}\overline{u}^{n+1}+g\left(\widehat{u}^{n+1}\right)=0,\\
			\displaystyle\frac{\ln V^{n+\frac12}-\ln V^{n-\frac12}}{\Delta t}=-\frac{\mathcal{K}(u^{n})}{E(u^{n})},\\
			\displaystyle u^{n+1}=\eta_k^{n+1}\overline{u}^{n+1}.
		\end{array}
	\end{equation}
	Here $\alpha_k$, $\beta_k$, $\widehat{u}^{n+1}$ and $\eta_k^{n+1}$ are different for $k$th-order schemes. For example, they can be defined as follows:
	
	BDF2:
	\begin{equation*}
		\aligned
		\alpha_2=\frac32,\quad \beta_2(u^n)=2u^n-\frac12u^{n-1},\quad \widehat{u}^{n+1}=2u^n-u^{n-1},\quad \eta_k^{n+1}=\chi(V^{n+\frac12}).
		\endaligned
	\end{equation*}
	
	BDF3:
	\begin{align*}
		\alpha_3&=\frac{11}{6},\quad \beta_3(u^n)=3u^n-\frac32u^{n-1}+\frac13u^{n-2},\\ \widehat{u}^{n+1}&=3u^n-3u^{n-1}+u^{n-2},\quad \eta_3^{n+1}=1-\left(1-\chi(V^{n+\frac12})\right)^2.
	\end{align*}
	
	BDF4:
	\begin{align*}
		\alpha_4&=\frac{25}{12},\quad \beta_4(u^n)=4u^n-3u^{n-1}+\frac43u^{n-2}-\frac14u^{n-3},\\ \widehat{u}^{n+1}&=4u^n-6u^{n-1}+4u^n-u^{n-3},\quad \eta_4^{n+1}=\eta_3^{n+1}.
	\end{align*}
	Here $\chi(V^{n+\frac12})$ is any explicit  approximation for 1 with $O(\Delta t^2)$, such as $\chi(V^{n+\frac12})=\frac{3V^{n+\frac12}-V^{n-\frac12}}{2E(\overline{u}^{n+1})}$.
\end{lemma}
\subsection{Staggered mesh scheme for general dissipative systems with unknown energy lower bound}
Actually for some dissipative systems, their energy lower bound is unknown which means that it is difficult to find a constant $C_0$ to keep $E_{tot}(u)+C_0>0$
 or we do not care whether the energy of the dissipative system is bounded or not. Hence we do not have to introduce the constant $C_0$ to obtain the energy bounded property. To construct an energy stable scheme, we redefine a new auxiliary variable $V(t)=E_{tot}(u)$ and take the time derivative of $V(t)$ to obtain the following energy dissipation law:
\begin{equation}\label{Modified-dissipative-energy-law3}
	\frac{dV}{dt}=\frac{dE_{tot}(u)}{dt}=-\mathcal{K}(u)\leq0.
\end{equation}

Considering that the derivative of the arctangent function is always positive, we take the time derivative of $\arctan V(t)$ to obtain
\begin{equation}\label{Modified-dissipative-energy-law4}
	\frac{d\arctan V}{dt}=\frac{1}{V^2+1}\frac{dV}{dt}=-\frac{\mathcal{K}(u)}{E_{tot}^2(u)+1}\leq0.
\end{equation}     

We combine the above equation \eqref{Modified-dissipative-energy-law4} with the first equation in the equivalent dissipative system \eqref{equivalent-dissipative-systems} to obtain the following new equivalent systems:
\begin{equation}\label{equivalent-dissipative-systems2}
	\left\{
	\begin{array}{ll}
		\displaystyle\frac{\partial u}{\partial t}+\mathcal{A}u+g(\eta u)=0,\\
		\displaystyle\frac{d\arctan V}{dt}=-\frac{\mathcal{K}(u)}{E_{tot}^2(u)+1},\\
		\displaystyle \eta(t)=\chi(V).
	\end{array}
	\right.
\end{equation}

Similarly as before, considering the Crank-Nicolson semi-implicit scheme for the first equation in \eqref{equivalent-dissipative-systems2} and discretizing the second equation in \eqref{equivalent-dissipative-systems2} on a staggered time meshes, we can also obtain the following CN-SM scheme:
\begin{equation}\label{CN-SM3}
	\begin{array}{l}
		\displaystyle\frac{u^{n+1}-u^{n}}{\Delta t}+\mathcal{A}\frac{u^{n+1}+u^{n}}{2}+g\left(\eta^{n+\frac12}\overline{u}^{n+\frac12}\right)=0,\\
		\displaystyle\frac{\arctan V^{n+
				\frac12}-\arctan V^{n-\frac12}}{\Delta t}=-\frac{\mathcal{K}(u^n)}{E_{tot}^2(u^n)+1},\\
		\displaystyle \eta^{n+\frac12}=\chi(V^{n+\frac12}).
	\end{array}
\end{equation}

It is easy to obtain the following results concerning the stability of the above schemes.
\begin{theorem}\label{theorem3}
	The new CN-SM scheme \eqref{CN-SM3} satisfies the following energy inequality 
	\begin{equation}\label{CN-SM4}
		V^{n+\frac12}-V^{n-\frac12}\leq0. 
	\end{equation}
\end{theorem}
\begin{proof}
From the second equation in CN-SM scheme \eqref{CN-SM3}, noting that $\mathcal{K}(u^n)\geq0$ and $\frac{1}{E_{tot}^2(u^n)+1}>0$, we easily get the following inequality  
	\begin{equation*}
		\arctan V^{n+\frac12}-\arctan V^{n-\frac12}=-\Delta t\frac{\mathcal{K}(u^n)}{E_{tot}^2(u^n)+1}\leq0.
	\end{equation*}
	Considering that the arctangent function is monotonically increasing, we immediately conclude $V^{n+\frac12}-V^{n-\frac12}\leq0$.
\end{proof}
\begin{remark}
	Noticing that the arctangent function has the range $(-\pi/2,\pi/2)$, we can introduce a small positive constant $\theta$ to redefine $V(t)=\theta E_{tot}$ to guarantee the right side of the second equation in scheme \eqref{CN-SM3} in the correct interval. The second equation in the new equivalent systems \eqref{equivalent-dissipative-systems2} will be rewritten as
	\begin{equation*}
		\frac{d\arctan V}{dt}=-\frac{\theta\mathcal{K}(u)}{\theta^2E_{tot}^2+1},
	\end{equation*}  
	and its discrete formula will become
	\begin{equation*}
		\displaystyle\frac{\arctan V^{n+
				\frac12}-\arctan V^{n-\frac12}}{\Delta t}=-\frac{\theta\mathcal{K}(u^n)}{\theta^2E_{tot}^2(u^n)+1}.
	\end{equation*} 
\end{remark}
\begin{remark}
In general, while models with unbounded free energy are mathematically admissible, they are typically deemed non-physical due to their violation of thermodynamic stability conditions. Consequently, the arctangent-type staggered mesh scheme \eqref{CN-SM3} proposed in this subsection serves as a mathematically consistent algorithm designed to maintain computational stability even when the energy lower bound is unknown. For systems with a well-defined energy lower bound $C_*$, we should update $V^{n+\frac12} = C_* $ in the staggered mesh scheme \eqref{CN-SM3} if $V^{n+\frac12}\leq C_* $ which means the dissipative system has reached a stable state. 
\end{remark}
\section{Error Analysis}
In this section, we shall establish error estimates for the semi-discrete version of the proposed staggered mesh scheme \eqref{CN-SM1}. We first consider the error estimate of the second equation corresponding to $V$ in \eqref{CN-SM1}.

Let us denote
\begin{equation*}
	\aligned
	& e_{u}^{n+1} = u^{n+1} - u(t^{n+1}), \  e_V^{n+\frac12} = V^{n+\frac12}-V(t^{n+\frac12}), \\
	& e_{\mathcal{K}}^{n+1} = \mathcal{K}(u^{n+1} ) - \mathcal{K}(u(t^{n+1} )), \ e_{E}^{n+1} = E(u^{n+1} ) - E(u(t^{n+1} )).
	\endaligned
\end{equation*} 

\medskip

\begin{theorem}\label{theorem_error}
	Suppose that $  u \in W^{3,\infty}(0,T; L^2(\Omega)) \cap W^{2,\infty}(0,T; H^2(\Omega)) $, $g(u) \in C^2(\mathbb{R})$, 
	$ \ln V(t)\in W^{3,\infty}(0,T) $ and $  V(t) \in W^{3,\infty}(0,T)  $, then there exists a positive constant $C$ independent of $\Delta t$ such that
	\begin{equation}\label{main result}
		\aligned
		\| e_u^{k+1} & \|^2 +  \Delta t \sum\limits_{n=1}^{k}  \| \mathcal{A} ^{1/2} \frac{e_u^{n+1}+  e_u^{n}}{2} \|^2 + |e_V^{k+\frac12}|^2 \\
		\leq & C  \Delta t\sum\limits_{n=1}^{k}\left|e_V^{n - \frac12}\right |^2  + C  \Delta t\sum\limits_{n=1}^{k}  |e_{E}^n | ^2 + C \Delta t \sum\limits_{n=1}^{k} \| e_u^{n}  \|^2
		+ \epsilon \Delta t\sum\limits_{n=1}^{k} |e_{\mathcal{K}}^n | ^2  + C(\Delta t)^4. 
		\endaligned
	\end{equation} 
\end{theorem}

\begin{proof}
	Though the error analysis for the unknown $u$ is rather standard and here for simplicity, we only demonstrate the main steps to obtain the final results. However,
the error estimates for the auxiliary numerical solution $V$ are nontrivial, so we give rigorous proof in the following lemma \ref{lemma1}.

	We first give the following \textit{a priori} assumption regarding the numerical solution 
	\begin{equation}\label{e_induction}
		\aligned
		\| u^n \|_{\infty} \leq C^*,
		\endaligned
	\end{equation} 
	where $C^*$ is a positive constant. A detailed proof for the \textit{a priori} assumption can be found by using a similar procedure as hypothesis (3.17a) in \cite{liu2023high} when combined with finite difference method, so here we skip that for simplicity. 
	
	Then we give error equation according to the first equation in \eqref{CN-SM3}:
	\begin{equation}\label{e_error u1}
		\aligned
		& \displaystyle\frac{e_u^{n+1} - e_u^{n}}{\Delta t}+\mathcal{A}\frac{e_u^{n+1}+ e_u^{n}}{2} +  g\left(\eta^{n+\frac12}\overline{u}^{n+\frac12}\right) -  g\left( u(t^{n+1/2} )\right) \\
		= & \displaystyle \frac{ \partial u(t^{n+1/2})}{ \partial t} - \frac{u(t^{n+1}) - u(t^{n})  }{\Delta t} + \mathcal{A} u(t^{n+1/2}) -  \mathcal{A} \frac{ u(t^{n+1}) + u(t^{n})} {2} \\
		= & - ( \frac{\partial^3 u(\zeta_1)}{\partial t^3} +\frac{\partial^3 u(\zeta_2)}{\partial t^3} )  \frac{ (\Delta t)^2 } {48} -
		 \mathcal{A}  ( \frac{ \partial ^2 u( \tilde{\zeta}_1) }{ \partial  t^2 } + \frac{ \partial ^2 u( \tilde{\zeta}_2) }{ \partial  t^2 } ) \frac{ (\Delta t)^2 } {8},
		\endaligned
	\end{equation} 
	where $ \zeta_i, \tilde{\zeta}_i  \in(t^{n-\frac12},t^{n+\frac12}), \ i=1,2$ and as stated in Remark \ref{rem_choose chi}, we choose $\eta^{n+\frac12} = \chi_1(V^{n+\frac12})=\frac{V^{n+\frac12}}{E(\overline{u}^{n+\frac12})}$ for simplicity. 
	Taking the inner product with $\frac{e_u^{n+1}+ e_u^{n}}{2}$, and noting the fact that $g(u) \in C^2(\mathbb{R})$ and \eqref{e_induction}, we have
	\begin{equation}\label{e_error u2}
		\aligned
		\frac{  \| e_u^{n+1}  \|^2 -  \| e_u^{n}  \|^2 }{ 2 \Delta t } + &  \| \mathcal{A} ^{1/2} \frac{e_u^{n+1}+ e_u^{n}}{2} \|^2 \leq C |e_V^{n+\frac12}|^2 + C |e_E^{n}|^2 \\
		& + C |e_E^{n-1}|^2  + C \| e_u^{n+1}  \|^2+ C \| e_u^{n}  \|^2 + C (\Delta t)^4,
		\endaligned
	\end{equation} 
	where we choose $C_0>0$ such that $E(u)=E_{tot}(u)+C_0>\frac{1}{2}$.

	Multiplying $\Delta t$ on both sides of \eqref{e_error u2} and summing up for $n=1,\ldots,k \ (k\leq N)$ lead to
	\begin{equation}\label{e_error u3}
		\aligned
		\| e_u^{k+1}  \|^2 + \Delta t \sum\limits_{n=1}^{k}  \| \mathcal{A} ^{1/2} \frac{e_u^{n+1}+ e_u^{n}}{2} \|^2 \leq & C  \Delta t\sum\limits_{n=1}^{k}\left|e_V^{n + \frac12}\right |^2 + C  \Delta t\sum\limits_{n=1}^{k}  |e_{E}^n | ^2  \\
		&+ C \Delta t\sum\limits_{n=1}^{k} \| e_u^{n+1}  \|^2 + C(\Delta t)^4. 
		\endaligned
	\end{equation} 
	
	To bound the terms on the right-hand side of \eqref{e_error u3}, we first give the following error estimates for the term $e_V$.
	
	\begin{lemma}\label{lemma1}
	Suppose that $  V(t) \in W^{3,\infty}(0,T) $  and $  \ln V(t)\in W^{3,\infty}(0,T) $, then there exists a positive constant $C$ independent of $\Delta t$ such that
	\begin{equation}\label{error-analysis6}
		\aligned
		\displaystyle|e_V^{k+\frac12}|^2\leq C\Delta t\sum\limits_{n=1}^{k}\left|e_V^{n-\frac12}\right |^2+ \epsilon \Delta t\sum\limits_{n=1}^{k}\left( |e_{\mathcal{K}}^n | ^2+ |e_{E}^n | ^2\right)+C(\Delta t)^4.
		\endaligned
	\end{equation} 
\end{lemma}

\begin{proof} We first establish the error equation corresponding to the auxiliary variable $V$. For the second equation in the equivalent system \eqref{equivalent-dissipative-systems}, by using the Taylor expansion with Lagrange residuals, we easily derive 
	\begin{equation}\label{error-analysis1}
	\aligned
		& \displaystyle\frac{\ln V(t^{n+\frac12})-\ln V(t^{n-\frac12})}{\Delta t} 
		=  -\frac{\mathcal{K}(u(t^n))}{E(u(t^n))} + R^{n+1/2}(V),
	\endaligned
	\end{equation}
	where $ R^{n+1/2}(V) = ( \frac{\partial^3\ln V( \hat{\zeta}_1)}{\partial t^3} +  \frac{\partial^3\ln V( \hat{\zeta}_2)}{\partial t^3} )
		 \frac{ (\Delta t)^2 } {48}, \quad \hat{\zeta}_i \in(t^{n-\frac12},t^{n+\frac12}), \ i=1,2.$
	From the properties of the logarithmic function, we easily have
	\begin{equation}\label{error-analysis2}
		\displaystyle\ln\frac{V(t^{n+\frac12})}{V(t^{n-\frac12})}= - \frac{\mathcal{K}(u(t^n))}{E(u(t^n))}\Delta t + R^{n+1/2}(V) \Delta t. 
	\end{equation}
	Applying the exponential function to both sides of the above equation simultaneously, we get the following equation
	\begin{equation}\label{error-analysis3}
		\displaystyle V(t^{n+\frac12})=e^{-\frac{\mathcal{K}(u(t^n))}{E(u(t^n))}\Delta t + R^{n+1/2}(V) \Delta t  }V(t^{n-\frac12}).
	\end{equation}
	Subtracting $V(t^{n-\frac12})$ from \eqref{error-analysis3} gives
	\begin{equation}\label{error-analysis4}
		\displaystyle V(t^{n+\frac12})-V(t^{n-\frac12})=\left(e^{-\frac{\mathcal{K}(u(t^n))}{E(u(t^n))}\Delta t + R^{n+1/2}(V) \Delta t }-1\right)V(t^{n-\frac12}).
	\end{equation}
	Subtracting \eqref{error-analysis4} from \eqref{CN-SM2} yields 
	\begin{equation}\label{error-analysis5}
		\aligned
		\displaystyle e_V^{n+\frac12}-e_V^{n-\frac12}
		&\displaystyle=\left(e^{-\frac{\mathcal{K}(u^n)}{E(u^n)}\Delta t}-1\right)V^{n-\frac12}
		-\left(e^{-\frac{\mathcal{K}(u(t^n))}{E(u(t^n))}\Delta t+ R^{n+1/2}(V) \Delta t }-1\right)V(t^{n-\frac12})\\
		&\displaystyle=\left(e^{-\frac{\mathcal{K}(u^n)}{E(u^n)}\Delta t}-e^{-\frac{\mathcal{K}(u(t^n))}{E(u(t^n))}\Delta t + R^{n+1/2}(V) \Delta t }\right)V^{n-\frac12}\\
		&\displaystyle\quad+\left(e^{-\frac{\mathcal{K}(u(t^n))}{E(u(t^n))}\Delta t + R^{n+1/2}(V) \Delta t }-1\right)\left(V^{n-\frac12}-V(t^{n-\frac12})\right).
		\endaligned
	\end{equation}
	Multiplying $e_V^{n+\frac12}$ on both sides of  \eqref{error-analysis5}, one derives immediately the following
	\begin{equation}\label{error-analysis7}
		\aligned
		&\displaystyle\frac{|e_V^{n+\frac12}|^2-|e_V^{n-\frac12}|^2}{2}+\frac{|e_V^{n+\frac12}-e_V^{n-\frac12}|^2}{2}
		\displaystyle\\
		=&\left(\left(e^{-\frac{\mathcal{K}(u^n)}{E(u^n)}\Delta t}-e^{-\frac{\mathcal{K}(u(t^n))}{E(u(t^n))}\Delta t + R^{n+1/2}(V) \Delta t }\right)V^{n-\frac12},e_V^{n+\frac12}\right)\\
		&\displaystyle+\left(\left(e^{-\frac{\mathcal{K}(u(t^n))}{E(u(t^n))}\Delta t + R^{n+1/2}(V) \Delta t }-1\right)e_V^{n-\frac12},e_V^{n+\frac12}\right).
		\endaligned
	\end{equation}
	
	For the first term on the right side of the above equation \eqref{error-analysis7}, noting that $-\frac{\mathcal{K}(u(t^n))}{E(u(t^n))}\Delta t\leq0$, then 
	 there exists a positive constant $C^*$ such that 
	\begin{equation}\label{error-analysis8}
		\aligned
		\displaystyle -\frac{\mathcal{K}(u(t^n))}{E(u(t^n))}\Delta t + R^{n+1/2}(V) \Delta t
		\leq C^*.
		\endaligned
	\end{equation}
	
	Using the Cauchy-Schwarz inequality, integral mean value theorem and \eqref{e_induction}, we can obtain that there exists a constant $\eta\leq C^*$ such that
	\begin{equation}\label{error-analysis9}
		\aligned
		&\displaystyle\left(\left(e^{-\frac{\mathcal{K}(u^n)}{E(u^n)}\Delta t}-e^{-\frac{\mathcal{K}(u(t^n))}{E(u(t^n))}\Delta t + R^{n+1/2}(V) \Delta t }\right)V^{n-\frac12},e_V^{n+\frac12}\right)\\
		&\displaystyle=\left(e^{\eta}\left(-\frac{\mathcal{K}(u^n)}{E(u^n)}\Delta t+\frac{\mathcal{K}(u(t^n))}{E(u(t^n))}\Delta t - R^{n+1/2}(V) \Delta t \right)V^{n-\frac12},e_V^{n+\frac12}\right)\\
		&\displaystyle \leq C\Delta t\left[\left|\frac{\mathcal{K}(u(t^n))}{E(u(t^n))}-\frac{\mathcal{K}(u^n)}{E(u^n)}\right|+\left|\frac{\partial^3\ln V}{\partial t^3}\right|\Delta t^2\right]\left|e_V^{n+\frac12}\right|\\
		&\displaystyle\leq C_1\Delta t\left|e_V^{n+\frac12}\right|^2+C\Delta t\left|\frac{\mathcal{K}(u(t^n))E(u^n)-\mathcal{K}(u^n)E(u(t^n))}{E(u(t^n))E(u^n)}\right|^2+C(\Delta t)^5\\
		&\displaystyle\leq C_1\Delta t\left|e_V^{n+\frac12}\right|^2+ \epsilon \Delta t\left(|\mathcal{K}(u(t^n))-\mathcal{K}(u^n)|^2+|E(u(t^n))-E(u^n)|^2\right)+C(\Delta t)^5,
		\endaligned
	\end{equation}
	where $ \epsilon $ is a positive constant which can be  chosen to be sufficiently small.

	For the second term on the right-hand side of the above equation \eqref{error-analysis7}, we find
	\begin{equation}\label{error-analysis10}
		\aligned
		&\displaystyle\left(\left(e^{-\frac{\mathcal{K}(u(t^n))}{E(u(t^n))}\Delta t + R^{n+1/2}(V) \Delta t }-1\right)e_V^{n-\frac12},e_V^{n+\frac12}\right)\\
		&=\displaystyle\left(\left(e^{-\frac{\mathcal{K}(u(t^n))}{E(u(t^n))}\Delta t + R^{n+1/2}(V) \Delta t }-e^0\right)e_V^{n-\frac12},e_V^{n+\frac12}\right)\\
		&=\displaystyle e^{\eta_2}\left(-\frac{\mathcal{K}(u(t^n))}{E(u(t^n))}\Delta t + R^{n+1/2}(V) \Delta t \right)\left(e_V^{n-\frac12},e_V^{n+\frac12}\right)\\
		&\leq\displaystyle \left(C\frac{\mathcal{K}(u(t^n))}{E(u(t^n))}\Delta t+C\Delta t^3\right)|\left(e_V^{n-\frac12},e_V^{n+\frac12}\right)|\\
		&\leq C_2\frac{\mathcal{K}(u(t^n))}{E(u(t^n))}\Delta t(\left|e_V^{n+\frac12}\right|^2+\left|e_V^{n-\frac12}\right|^2)+C\Delta t (\left|e_V^{n+\frac12}\right|^2+\left|e_V^{n-\frac12}\right|^2).
		\endaligned
	\end{equation}
	
	Assuming that the true solution $|u(t^n)|\leq C$, we combine the above inequalities \eqref{error-analysis9}-\eqref{error-analysis10} into \eqref{error-analysis7}, we arrive at
	\begin{equation}\label{error-analysis11}
		\aligned
		&\displaystyle\frac{|e_V^{n+\frac12}|^2-|e_V^{n-\frac12}|^2}{2}+\frac{|e_V^{n+\frac12}-e_V^{n-\frac12}|^2}{2}\\
		\leq& \left(C_1+C_2\frac{\mathcal{K}(u(t^n))}{E(u(t^n))}+C\Delta t^2\right)\Delta t\left|e_V^{n+\frac12}\right|^2+C\Delta t\left|e_V^{n-\frac12}\right|^2\\
		&+\epsilon \Delta t\left(|e_{\mathcal{K}}^n|^2+|e_{E}^n|^2\right)+C(\Delta t)^5\\
		\leq &C_3\Delta t\left|e_V^{n+\frac12}\right|^2+C\Delta t\left|e_V^{n-\frac12}\right|^2+\epsilon \Delta t\left(|e_{\mathcal{K}}^n|^2+|e_{E}^n|^2\right)+C(\Delta t)^5.
		\endaligned
	\end{equation}
	
	Supposing $C_3\Delta t\leq\frac14$ and  summing up above inequality for $n=1,\ldots,k \ (k\leq N)$, we obtain
	\begin{equation}\label{error-analysis12}
		\aligned
		\displaystyle|e_V^{k+\frac12}|^2+& \sum\limits_{n=1}^{k}|e_V^{n+\frac12}-e_V^{n-\frac12}|^2\\
		&\leq\left|e_V^{\frac12}\right|^2+C\Delta t\sum\limits_{n=1}^{k}\left|e_V^{n-\frac12}\right|^2+\epsilon \Delta t\sum\limits_{n=1}^{k}\left(|e_{\mathcal{K}}^n|^2+|e_{E}^n|^2\right)+C(\Delta t)^4.
		\endaligned
	\end{equation}
	
	Denoting that $V^{\frac12}=E(\overline{u}^{\frac12})$ and $u(t^{\frac12})-\overline{u}^{\frac12}$ has a local truncation error of $O(\Delta t^2)$, we immediately have $\left|e_V^{\frac12}\right|^2=O(\Delta t^4)$. Combining it with above inequality and dropping some unnecessary terms, we arrive at 
	\begin{equation*}
		\aligned
		\displaystyle|e_V^{k+\frac12}|^2\leq C\Delta t\sum\limits_{n=1}^{k}\left|e_V^{n-\frac12}\right|^2+\epsilon \Delta t\sum\limits_{n=1}^{k}\left(|e_{\mathcal{K}}^n|^2+|e_{E}^n|^2\right)+C(\Delta t)^4,
		\endaligned
	\end{equation*}
	which completes the proof.
\end{proof}	
	
	Then combining  \eqref{e_error u3} with \eqref{error-analysis6} in Lemma \ref{lemma1}, we have
	\begin{equation}\label{e_error u4}
		\aligned
		\| e_u^{k+1} & \|^2 +  \Delta t \sum\limits_{n=1}^{k}  \| \mathcal{A} ^{1/2} \frac{e_u^{n+1}+  e_u^{n}}{2} \|^2 + |e_V^{k+\frac12}|^2 \\
		\leq & C  \Delta t\sum\limits_{n=1}^{k}\left|e_V^{n - \frac12}\right |^2  + C  \Delta t\sum\limits_{n=1}^{k}  |e_{E}^n | ^2 + C \Delta t \sum\limits_{n=1}^{k} \| e_u^{n}  \|^2
		+ \epsilon \Delta t\sum\limits_{n=1}^{k} |e_{\mathcal{K}}^n | ^2  + C(\Delta t)^4,
		\endaligned
	\end{equation} 
	which leads to the desired result.
	
\end{proof}
\begin{remark}
	Please note that here we only consider the error estimates for general dissipative systems. We need to estimate the error of $|e_{\mathcal{K}}^n |$ and $ |e_{E}^n | ^2$ when we consider the particular system, including Allen-Cahn equation, Cahn-Hilliard equation and Navier-Stokes equations, et al. Specially for the Navier-Stokes equations, one should take the inner product with $\frac{e_u^{n+1}- e_u^{n}}{ \Delta t}$ in \eqref{e_error u1} to control the term $\| \nabla   e_u^{n} \|^2$ on the right-hand side. In addition,  a modification to the Crank-Nicolson method can also be used in \eqref{CN-SM3} to obtain the desired error estimate, i.e.
	\begin{equation}\label{CN-SM3_modification}
		\begin{array}{l}
			\displaystyle\frac{u^{n+1}-u^{n}}{\Delta t}+\mathcal{A}(\frac{3}{4}u^{n+1}+\frac{1}{4}u^{n-1})+g\left(\eta^{n+\frac12}\overline{u}^{n+\frac12}\right)=0,\\
			\displaystyle\frac{\ln V^{n+
					\frac12}-\ln V^{n-\frac12}}{\Delta t}=-\frac{\mathcal{K}(u^n)}{E(u^n)},\\ 
			\displaystyle\eta^{n+\frac12}=\mathcal{F}(V^{n+\frac12}).
		\end{array}
	\end{equation}
	
	Hence by using similar error estimates of Lemma 2.12 in \cite{diegel2017convergence}, the nonlinear term of $|e_{\mathcal{K}}^n |$ can be controlled by $\| \nabla ( \frac{3}{4}u^{n+1} +   \frac{1}{4}u^{n-1}) \| $. 
\end{remark}

\section{Numerical Simulations}
In this section, we present several numerical experiments to verify the theoretical results of the proposed schemes. Unless otherwise specified, periodic boundary conditions are implemented for all examples and  $\chi(V)=\chi_1(V)$, with the computations being conducted using a Fourier spectral method.
\subsection{Convergence tests}
In this example, we first verify the accuracy of the proposed numerical schemes for the Allen-Cahn equation, specifically,
$$\mathcal{A}= -\Delta,\quad g(u)=\frac{1}{\epsilon^2}(u^3-u),$$
as well as for the Cahn-Hilliard equation, given by
$$\mathcal{A}=\Delta^2,\quad g(u)=-\frac{1}{\epsilon^2}\Delta(u^3-u).$$ 
We set the parameter $\epsilon=0.7$ 
and assume that the exact solution is described by
\begin{equation*}
	u(x,y,t)=\cos(x)\cos(y)\sin{t}.
\end{equation*}
To discretize the spatial variables, we define the domain as $\Omega= [0, 2\pi)^2$ and use a $256\times256$ grid, ensuring that the spatial discretization error is negligible compared to the temporal discretization error. Figs. \ref{newSAVerror1}-\ref{newSAVerror2} present the $L^2$ errors between the numerical and exact solutions at $T=1$.  It is evident that all schemes achieve the expected accuracy in time. In comparison to the Lagrange Multiplier (LM) scheme presented in  \cite{cheng2020new} and the 
general scalar auxiliary variable (GSAV) scheme  discussed in \cite{huang2022new}, our CN-SM scheme \eqref{CN-SM1} demonstrates significantly smaller errors as shown in Figs. \ref{newSAVerror1}(a)-\ref{newSAVerror2}(a).
Regarding energy approximation, the constructed scheme  achieve second-order accuracy, while the GSAV method only attains first-order accuracy, as illustrated in Figs. \ref{newSAVerror1}(b)-\ref{newSAVerror2}(b).

\begin{figure}[!t]
	\centering 
	\subfloat[Time step size vs. errors of $u$]{
		\begin{minipage}[c]{0.45\textwidth}
			\includegraphics[width=1\textwidth]{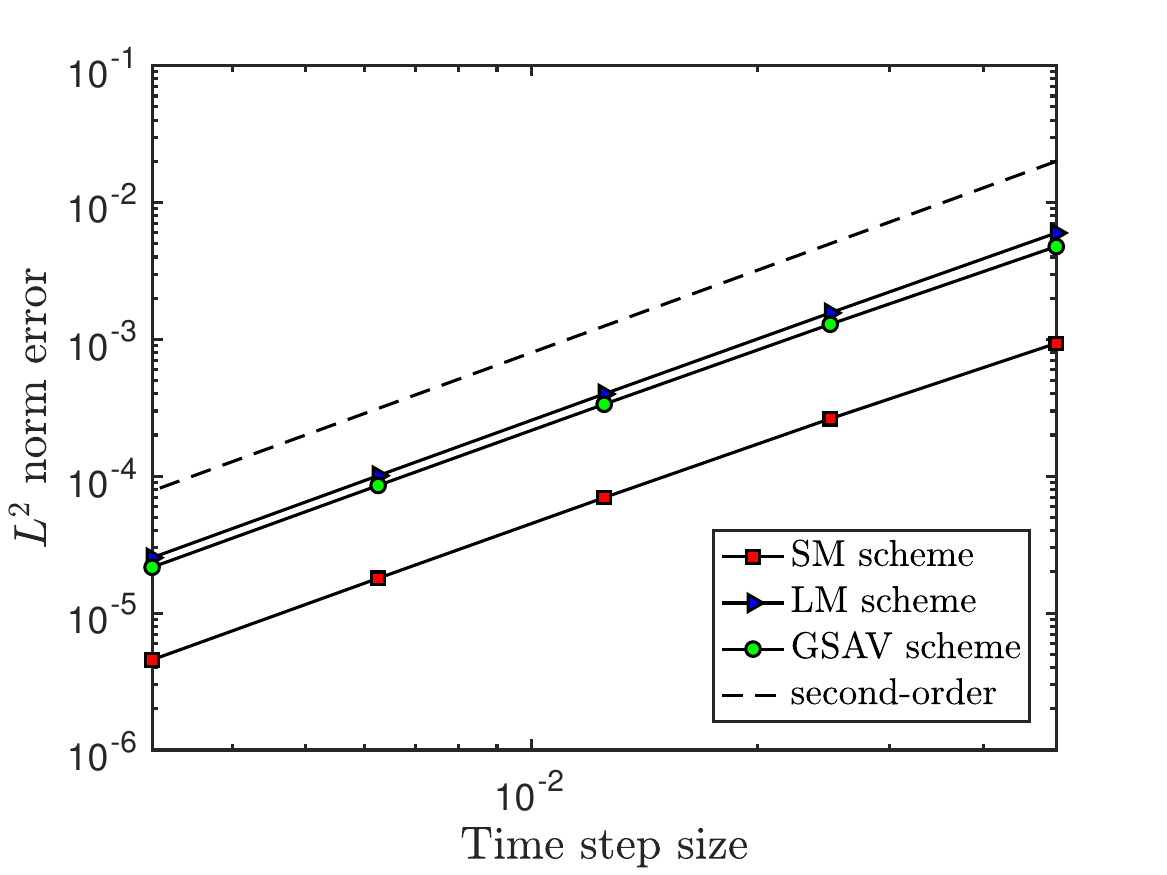}\label{newSAVerror1a}
		\end{minipage}
	}
	\subfloat[Time step size  vs. errors of $E$]{
		\begin{minipage}[c]{0.45\textwidth}
			\includegraphics[width=1\textwidth]{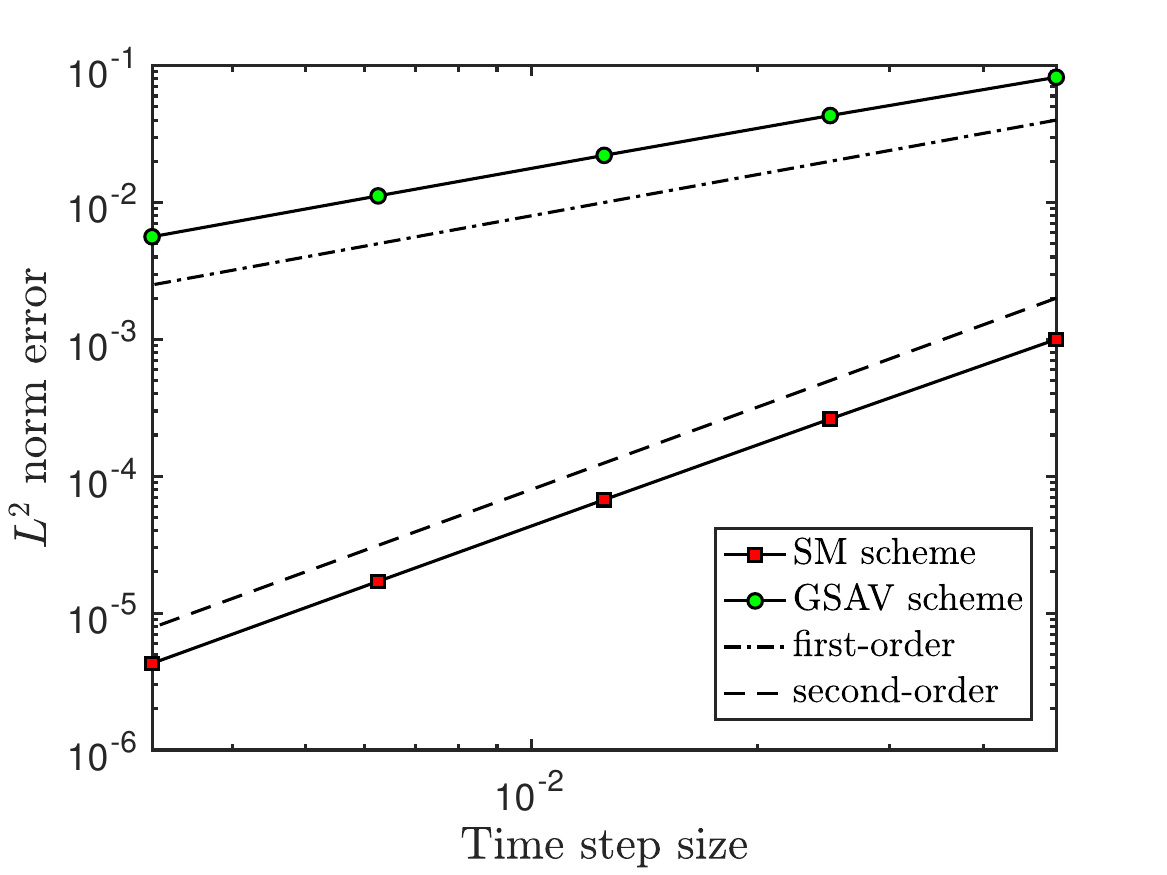}\label{newSAVerror1b}
		\end{minipage}
	}
	\caption{Numerical convergence  rates for the Allen-Cahn equation.
	}\label{newSAVerror1}
\end{figure}

\begin{figure}[!t]
	\centering 
	\subfloat[Time step size vs. errors of $u$]{
		\begin{minipage}[c]{0.45\textwidth}
			\includegraphics[width=1\textwidth]{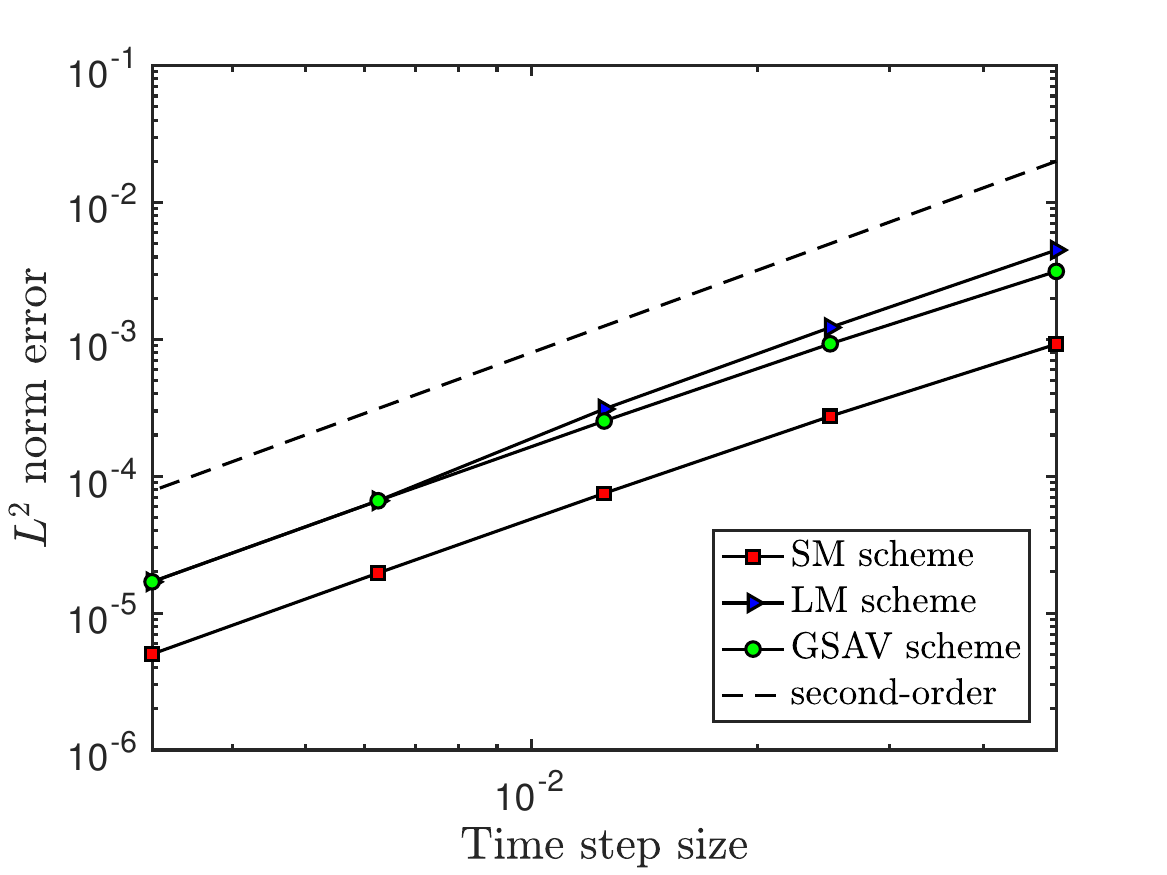}\label{newSAVerror2a}
		\end{minipage}
	}
	\subfloat[Time step size  vs. errors of $E$]{
		\begin{minipage}[c]{0.45\textwidth}
			\includegraphics[width=1\textwidth]{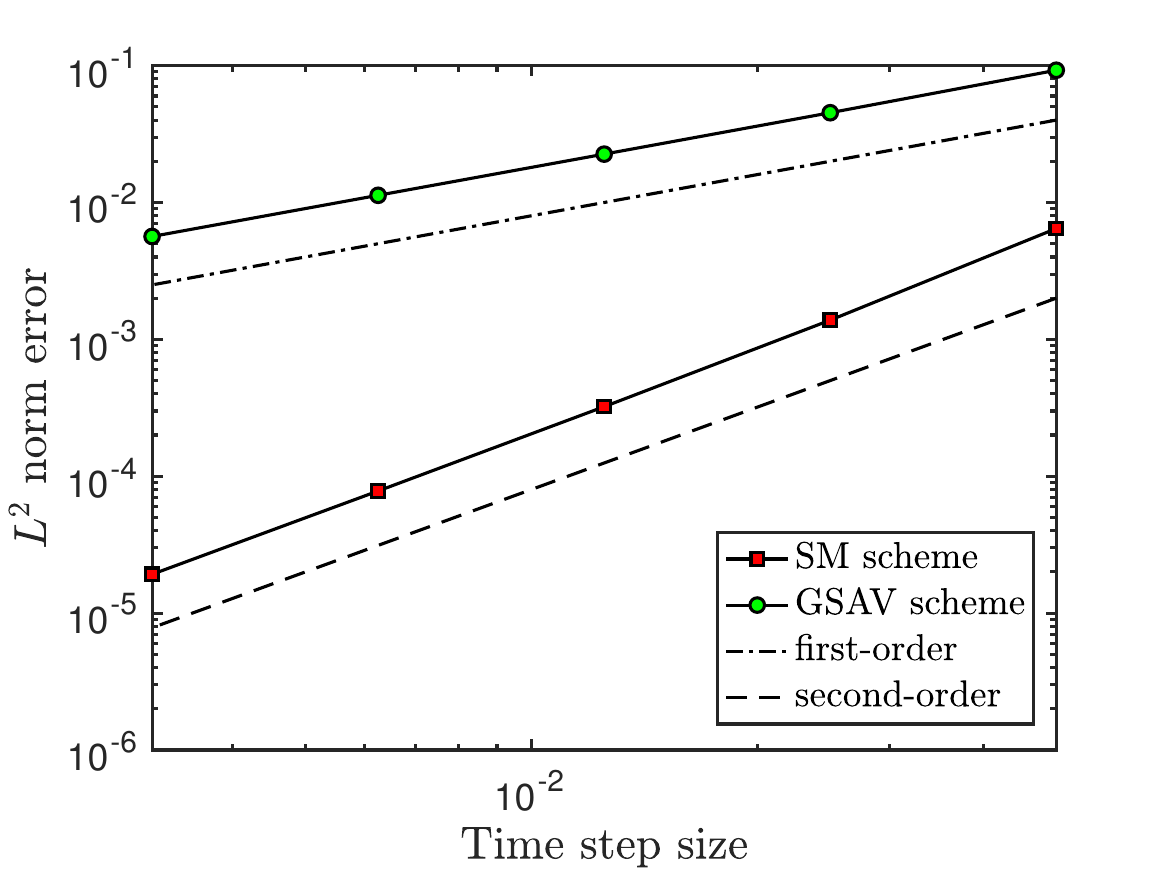}\label{newSAVerror2b}
		\end{minipage}
	}
	\caption{Numerical convergence rates for the Cahn-Hillard equation.
	}\label{newSAVerror2}
\end{figure}
\subsection{Navier-Stokes equations}\label{NSeq}
In this section, we present a numerical investigation of the Navier-Stokes equations using the CN-SM scheme. The Navier-Stokes equations are fundamental to modeling fluid dynamics and can be expressed as follows:
\begin{equation}\label{NS_equation}
	\begin{array}{l}
		\textbf{u}_t-\nu\Delta \textbf{u}+(\textbf{u}\cdot\nabla)\textbf{u}+\nabla p=0,\\
		\nabla\cdot \textbf{u}=0.
	\end{array}
\end{equation}
This system \eqref{NS_equation} satisfies a specific principle of energy dissipation, which can be written as follows:
\begin{equation*}
	\frac{dE(\textbf{u})}{dt}=-\nu\|\nabla \textbf{u}\|^2,
\end{equation*}
where $E(\textbf{u})$ is defined as the kinetic energy given by $E(\textbf{u})=\frac{1}{2}\|\textbf{u}\|^2$.
We also introduce an auxiliary variable defined as $V(t)=E(\textbf{u})+C_0$, where $C_0 \geq 0$ ensures $V(t)>0$. This approach allows us to construct an unconditionally stable numerical scheme, which is implemented effectively as follows:
\begin{equation}\label{CN-SM-NS}
	\begin{array}{l}
		\displaystyle\frac{\textbf{u}^{n+1}-\textbf{u}^{n}}{\Delta t}-\nu\Delta\frac{\textbf{u}^{n+1}+\textbf{u}^{n}}{2}+(	\eta^{n+\frac12}\overline{\textbf{u}}^{n+\frac12}\cdot\nabla)	\eta^{n+\frac12}\overline{\textbf{u}}^{n+\frac12}+	\nabla p^{n+\frac12}=0,\\
		\nabla\cdot \textbf{u}^{n+\frac12}=0,\\
		\displaystyle\frac{\ln V^{n+
				\frac12}-\ln V^{n-\frac12}}{\Delta t}=-\frac{\mathcal{K}(u^n)}{E(u^n)},\\
		\displaystyle\eta^{n+\frac12}=\chi(V^{n+\frac12}).
	\end{array}
\end{equation}
This formulation ensures numerical stability and allows for efficient computational implementation.
\subsubsection{Convergence tests}
We first perform a test where the right-hand side is derived from the following analytical solution:
\begin{equation*}
	\begin{aligned}
		\textbf{u}_e(x,y,t)&= \pi\sin(t)(\sin(2\pi x)\cos(2\pi y),-\cos(2\pi x)\sin(2\pi y))^{\mathrm T},\\
		p_e(x,y,t)&=\sin(t)\cos(2\pi x)\sin(2\pi y).
	\end{aligned}
\end{equation*}
In this simulation, we consider the computational domain defined as $\Omega= [0, 1)^2$ and set the kinematic viscosity $\nu= 1$. We implement a spatial discretization using $ 256\times 256$ Fourier modes. The $L^2$-norm errors for the CN-SM scheme are illustrated in Fig. \ref{NS_periodic2}(a), confirming the expected convergence rates of the numerical method.
\subsubsection{Double shear layer problem}
Next, we investigate the double shear layer problem. The initial conditions for the velocity field $\textbf{u}=(u_1,u_2)$ are defined as follows:
\begin{equation}
	\begin{aligned} 
		& u_1(x, y, 0)=\left\{\begin{array}{l}\tanh (\rho(y-0.25)), y \leq 0.5, \\ \tanh (\rho(0.75-y)), y>0.5,\end{array}\right. \\
		& u_2(x, y, 0)=\sigma \sin (2 \pi x).
	\end{aligned}
\end{equation}
We focus on the double shear layer problem within the framework of the Navier-Stokes equations.
Here, the parameter  $\rho$ signifies the width of the shear layer, while  $\sigma$ indicates the magnitude of the perturbation. For our simulations, we set $\sigma= 0.05$  and consider a computational domain 
$\Omega= [0, 1)^2$.
We present the evolution of the vorticity contours for the parameters  $\rho = 100$,  $\nu =5 \times 10^{-5}$, $N =256$, and a time step  of $\Delta t =3 \times 10^{-4}$, computed using the CN-SM scheme, as depicted in Fig. \ref{NS_periodic}.
Additionally, in Fig. \ref{NS_periodic2}(b)-(c), we provide the temporal evolution of the total energy and the parameter $\eta$ for the CN-SM scheme. The results indicate that the total energy exhibits a decay trend, while 
$\eta$ oscillates around the value of $1$, remaining positive throughout the simulation.
The results demonstrate that the vortex  intensity progressively increases over time.  
Notably, our second-order scheme yields accurate solutions, in contrast to the second-order scheme presented in \cite{huang2021stability}, which produced erroneous results.

\begin{figure}[!t]
	\centering 
	\subfloat[]{
		\begin{minipage}[c]{0.31\textwidth}
			\includegraphics[width=1\textwidth]{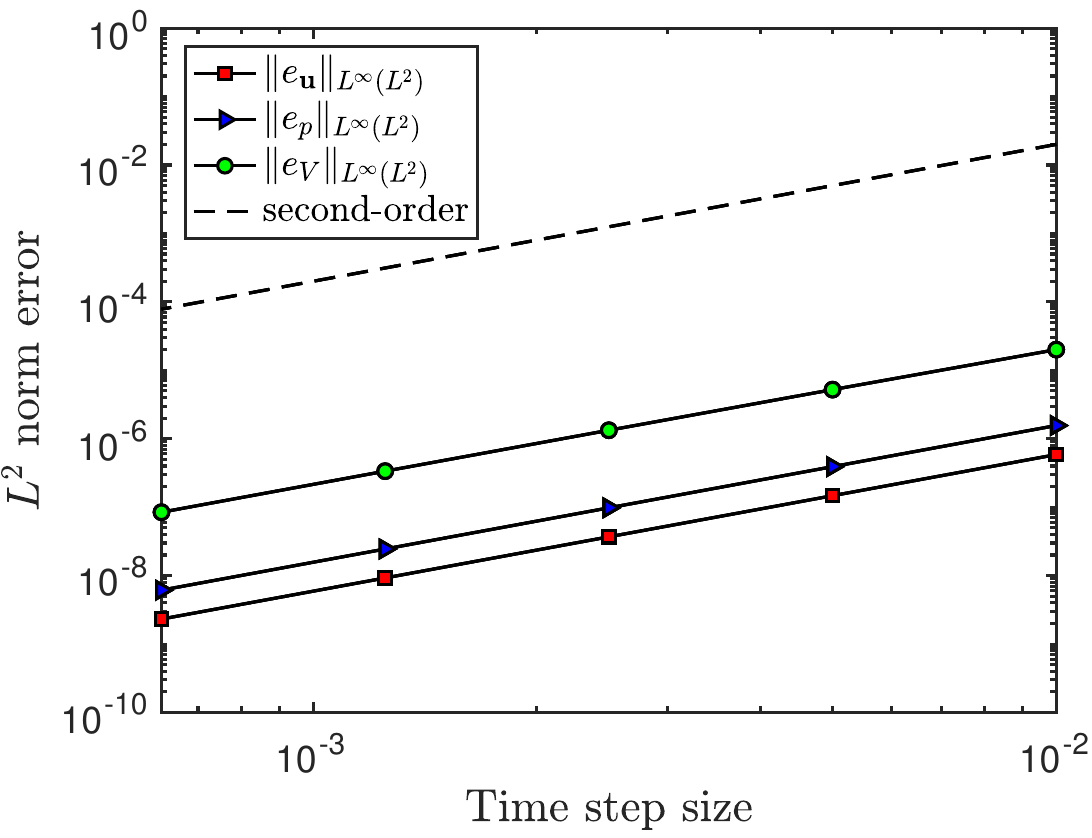}
		\end{minipage}
	}
	\subfloat[]{
		\begin{minipage}[c]{0.31\textwidth}
			\includegraphics[width=1\textwidth]{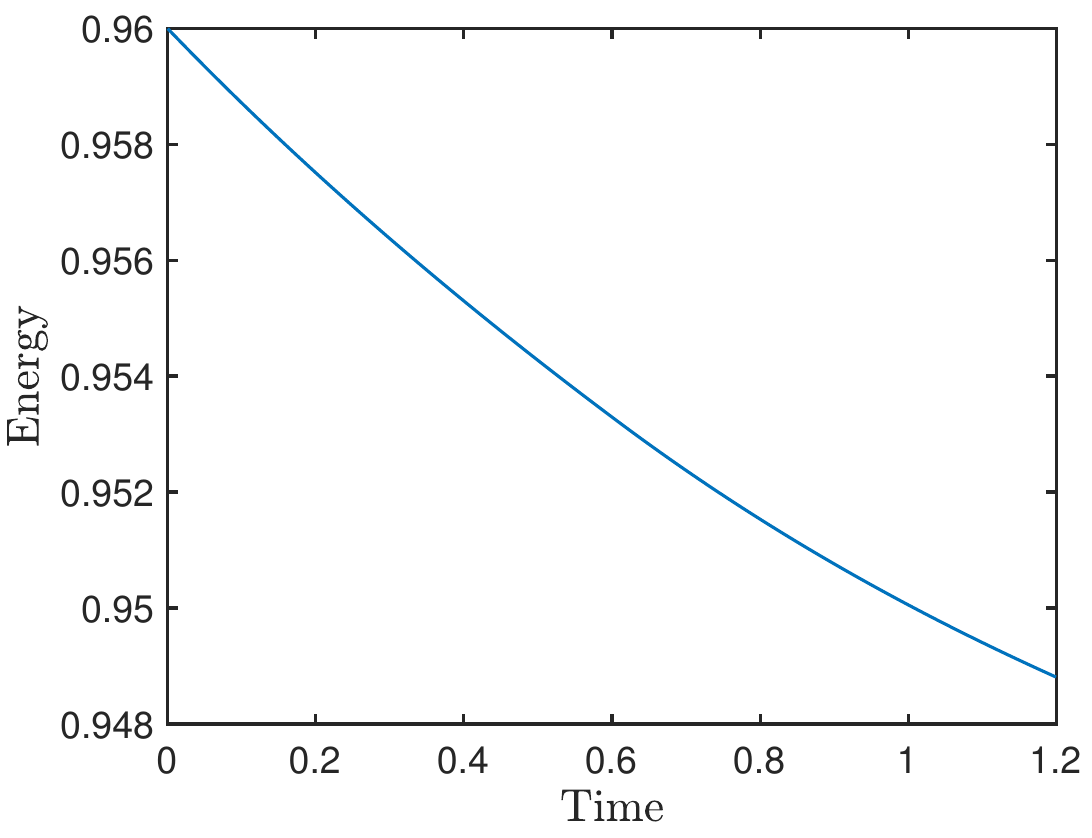}
		\end{minipage}
	}
	\subfloat[]{
		\begin{minipage}[c]{0.31\textwidth}
			\includegraphics[width=1\textwidth]{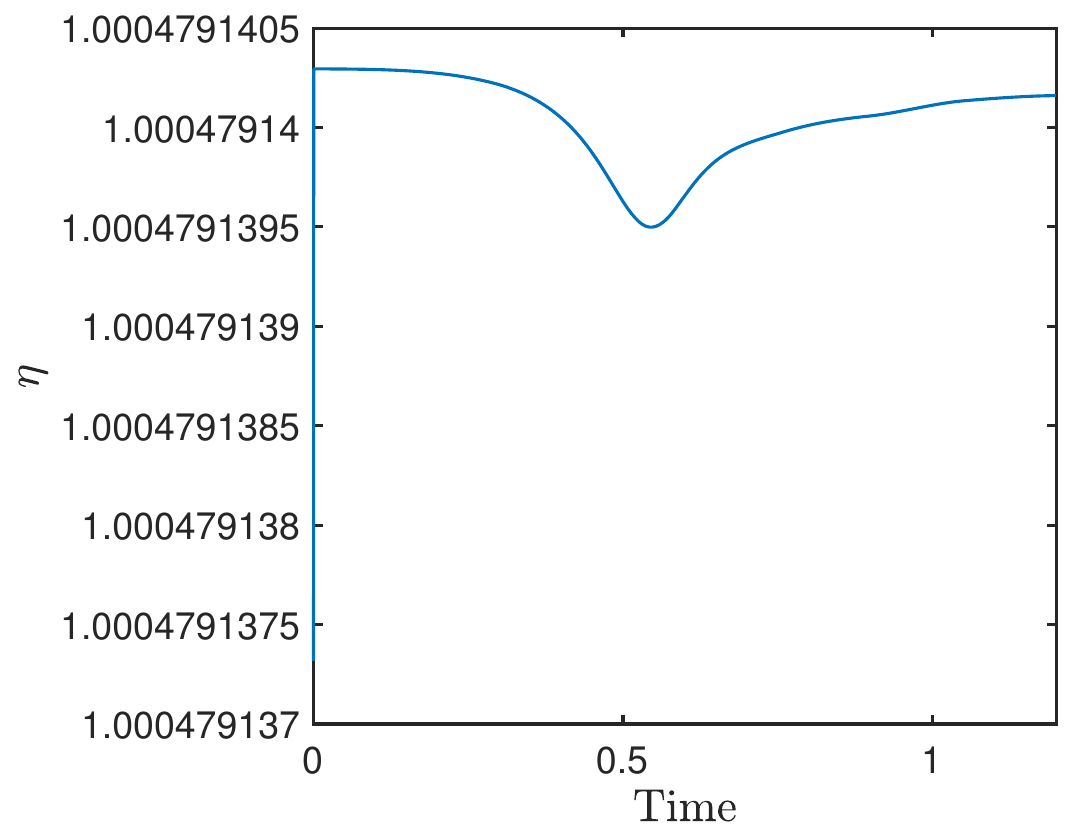}
		\end{minipage}
	}
	\caption{(a): Numerical convergence rates of the Navier-Stokes equations; (b)-(c): Evolution of energy and parameter $\eta$ with $\Delta t=3 \times 10^{-4}$, \ $\rho=100$,\ $\nu=5 \times 10^{-5}$ and $\sigma=0.05$.
	}\label{NS_periodic2}
\end{figure}

\begin{figure}[!t]
	\centering 
			\includegraphics[width=0.32\textwidth]{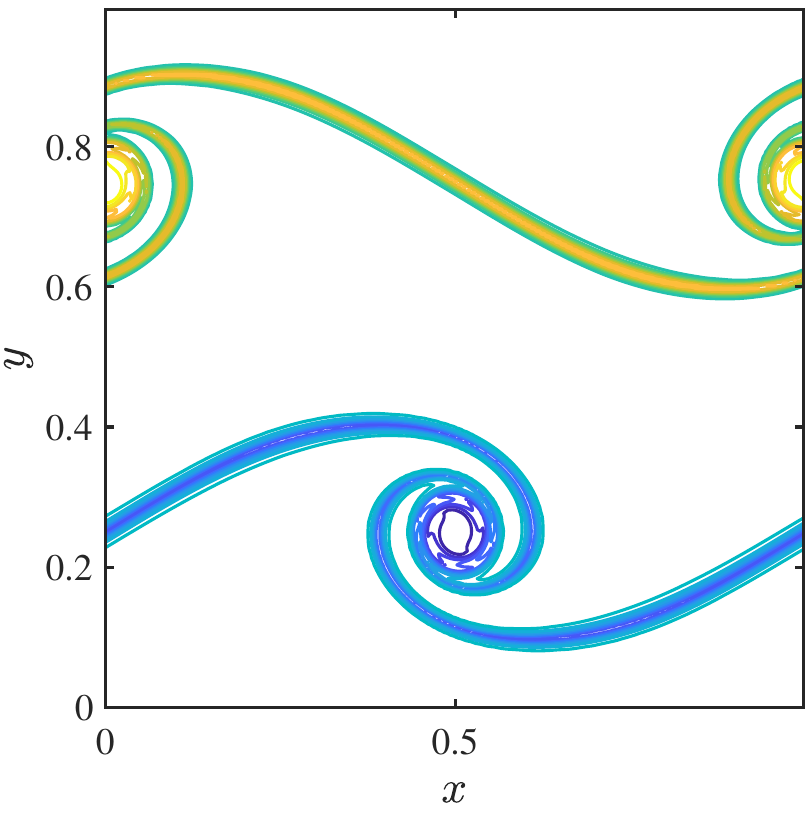}
			\includegraphics[width=0.32\textwidth]{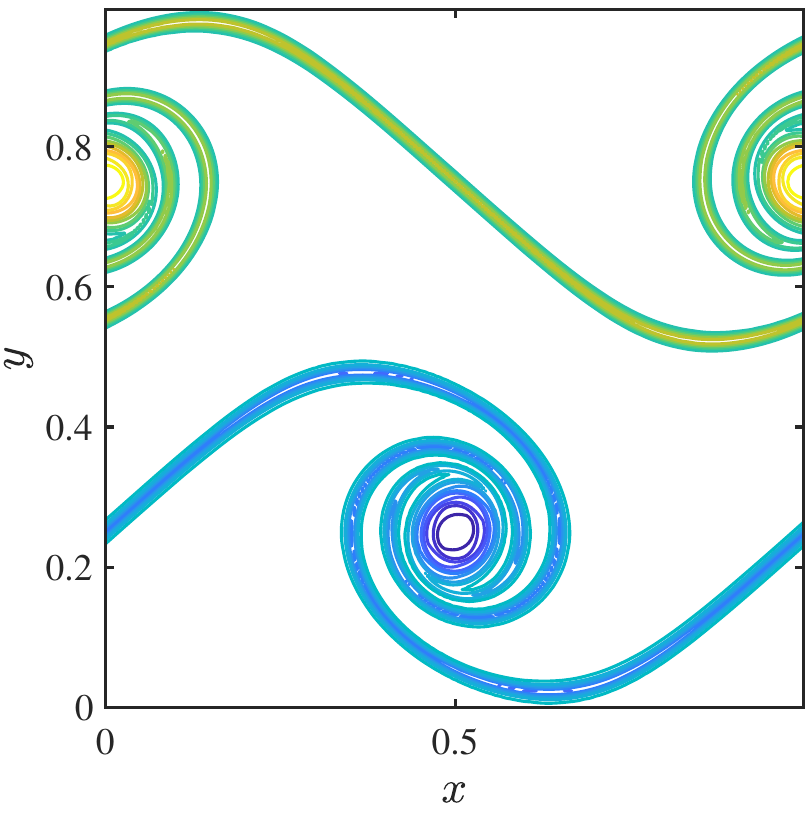}
			\includegraphics[width=0.32\textwidth]{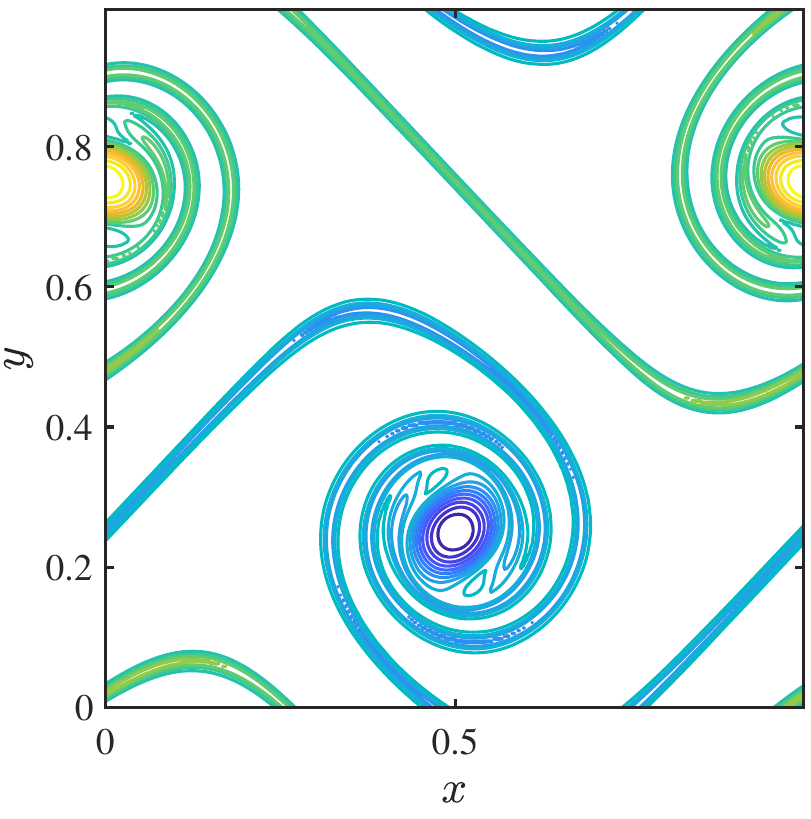}
	\caption{The evolution of vorticity of Navier-Stokes equations with $\Delta t=3 \times 10^{-4}$, \ $\rho=100$,\ $\nu=5 \times 10^{-5}$,\ $\sigma=0.05$ at $T = 0.8, \ 1,\ 1.2$.
	}\label{NS_periodic}
\end{figure}

\subsubsection{Kelvin-Helmholtz instability problem}
The initial velocity difference across a shear layer can lead to the growth of small disturbances, resulting in the formation of vortices, a phenomenon known as Kelvin-Helmholtz instability.  
We investigate a unit square computational domain with periodic boundary conditions at $x=0$ and $x=1$. At $y=0$ and $y=1$, we apply a no-penetration boundary condition, $\textbf{u} \cdot \textbf{n} =0$, and a weakly enforced free-slip condition, $(-\nu \nabla \textbf{u} \cdot \textbf{n} ) \times \textbf{n} =0$. The initial velocity field is given by:
$$
\textbf{u}(x, y, 0)=\binom{\tanh (28(2 y-1))}{0}+c_n\binom{\partial_y \psi(x, y)}{-\partial_x \psi(x, y)},
$$
with the stream function defined as:
$$
\psi(x, y)=\exp \left(-28^2(y-0.5)^2\right)(\cos (8 \pi x)+\cos (20 \pi x)) .
$$
Here,  $\delta=\frac{1}{28}$ represents the initial vorticity thickness, and $c_n=10^{-3}$ is a scaling factor.
Given a fixed Reynolds number of $Re = 100$, the kinematic viscosity is determined to be $\nu=\delta/Re=1/2800$. We set $N=200$ and $\Delta t=5e-4$ for the simulations by using the Taylor-Hood finite element (P2-P1) for the velocity and pressure.
The evolutions of vorticity generated by the scheme \eqref{CN-SM-NS} at various time intervals are depicted in Fig. \ref{KH}. It is observed that four vortices gradually emerge from the initial conditions, exhibiting instability and a tendency to merge into two larger vortices. These two ellipsoidal vortices ultimately combine into a single, larger vortex by the final time $T=8$,  which aligns with the benchmark results presented in
\cite{john2005assessment,schneider2000numerical,schroeder2019reference}.
\begin{figure}[!t]
	\centering 
	\includegraphics[width=0.32\textwidth]{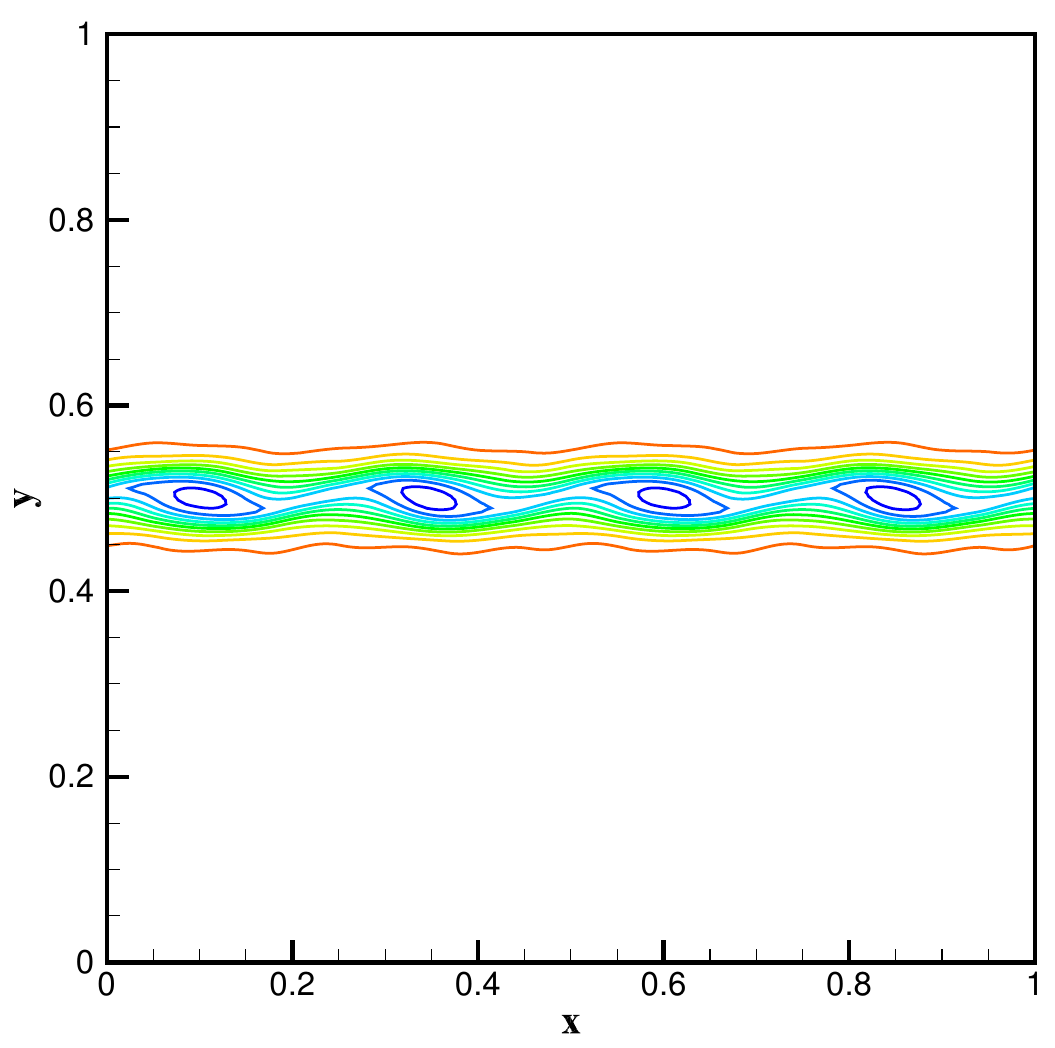}
	\includegraphics[width=0.32\textwidth]{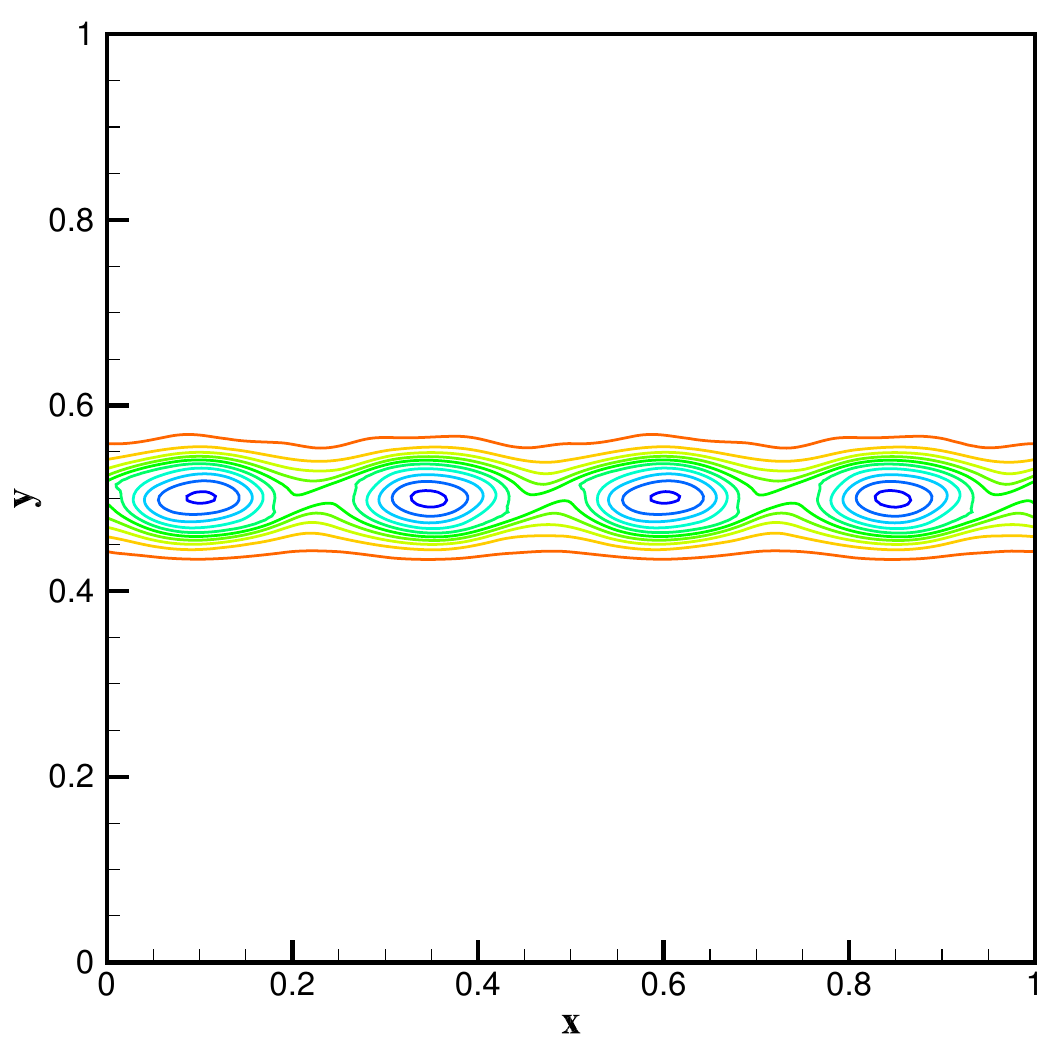}
	\includegraphics[width=0.32\textwidth]{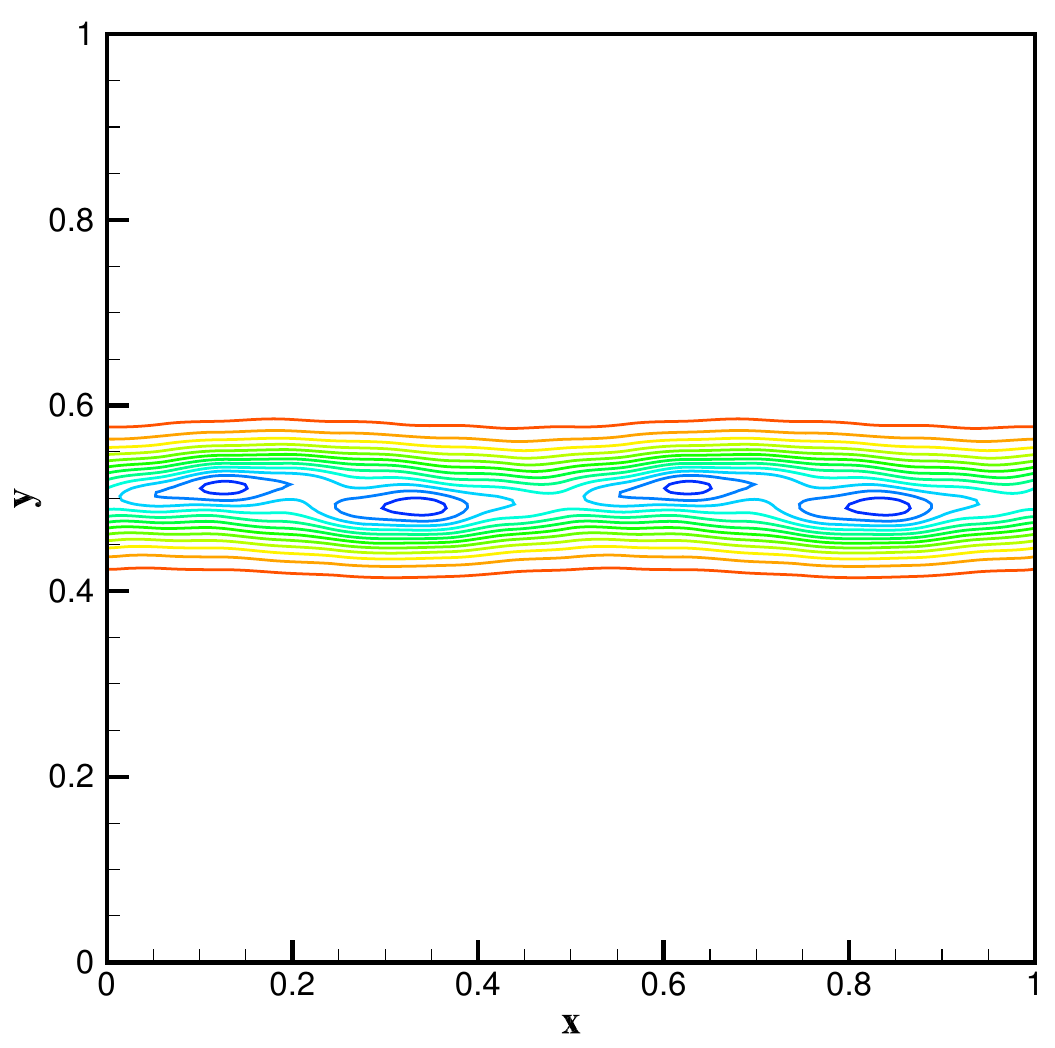}
	\includegraphics[width=0.32\textwidth]{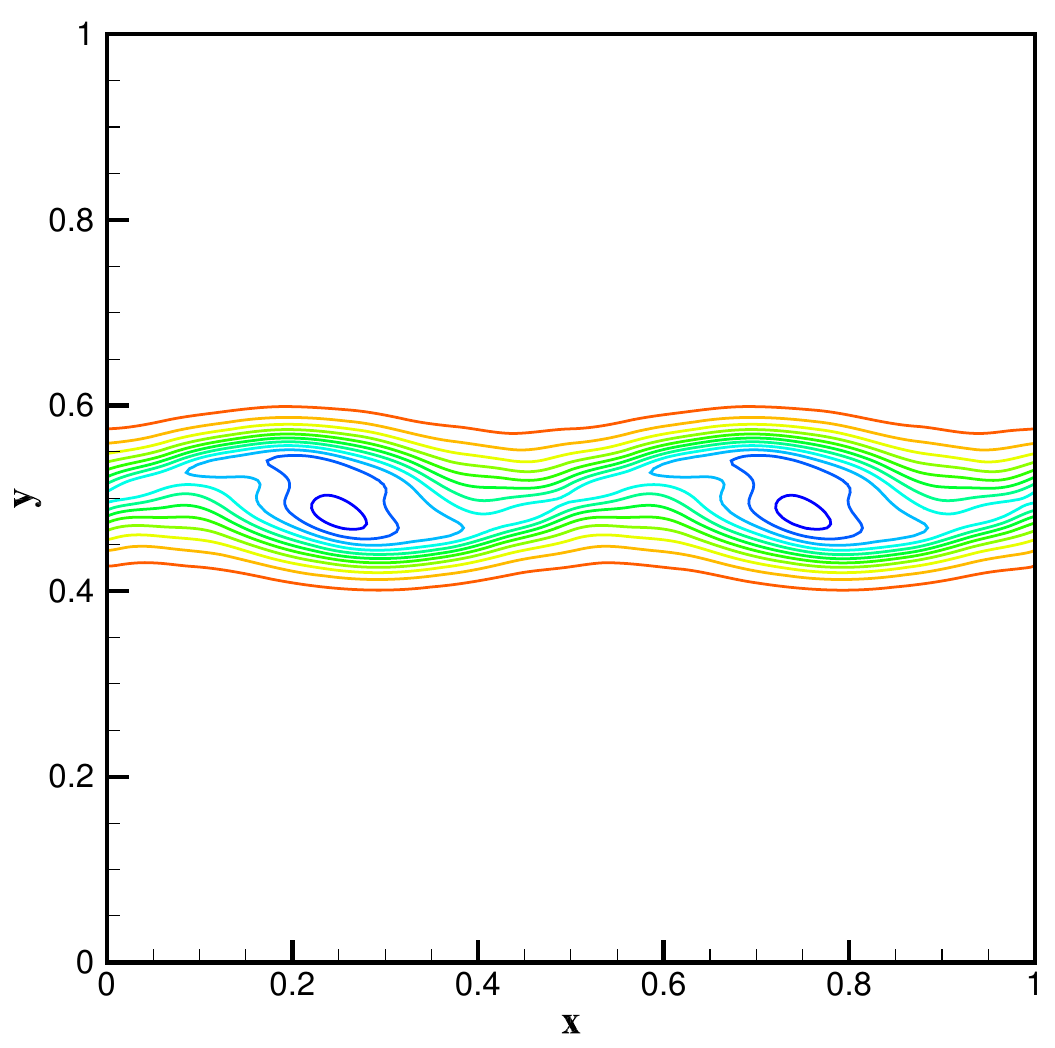}
	\includegraphics[width=0.32\textwidth]{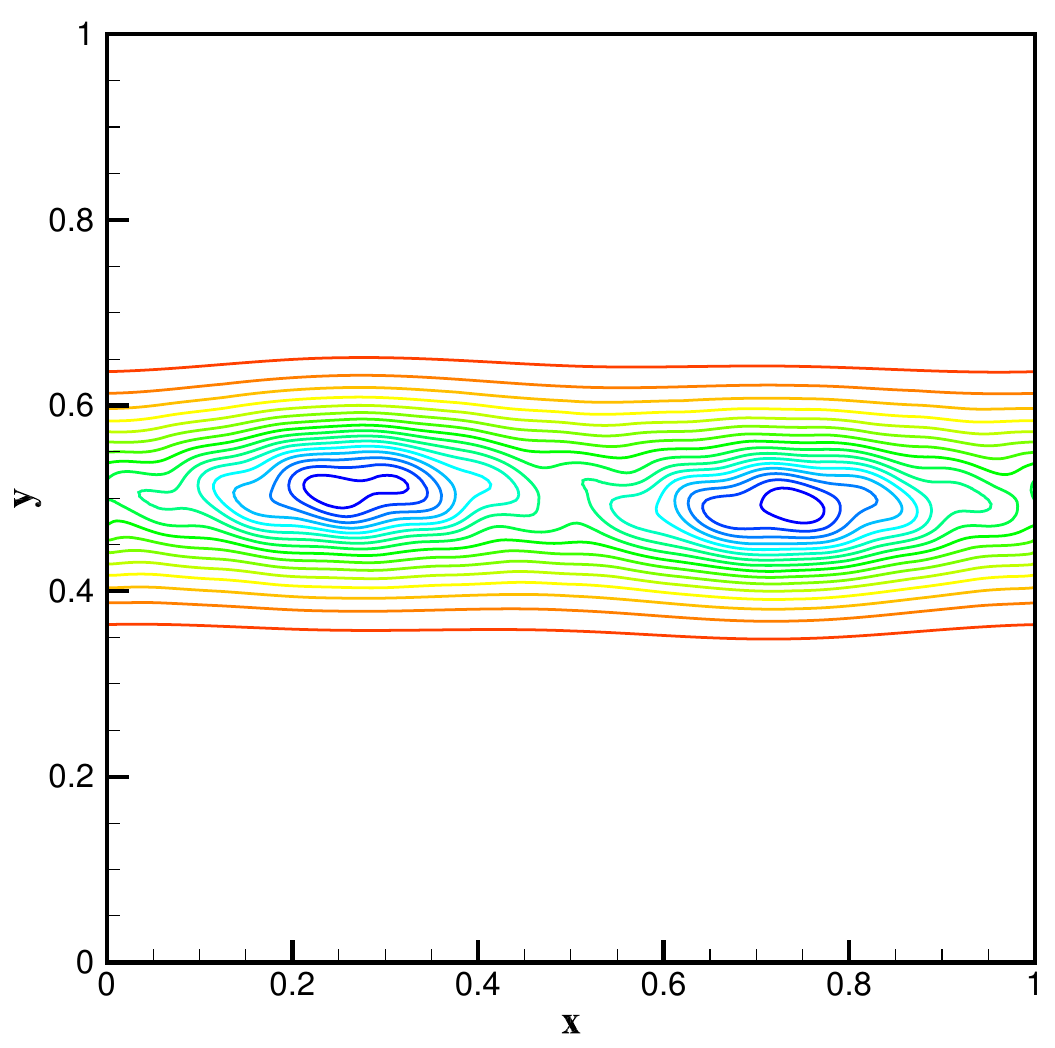}
	\includegraphics[width=0.32\textwidth]{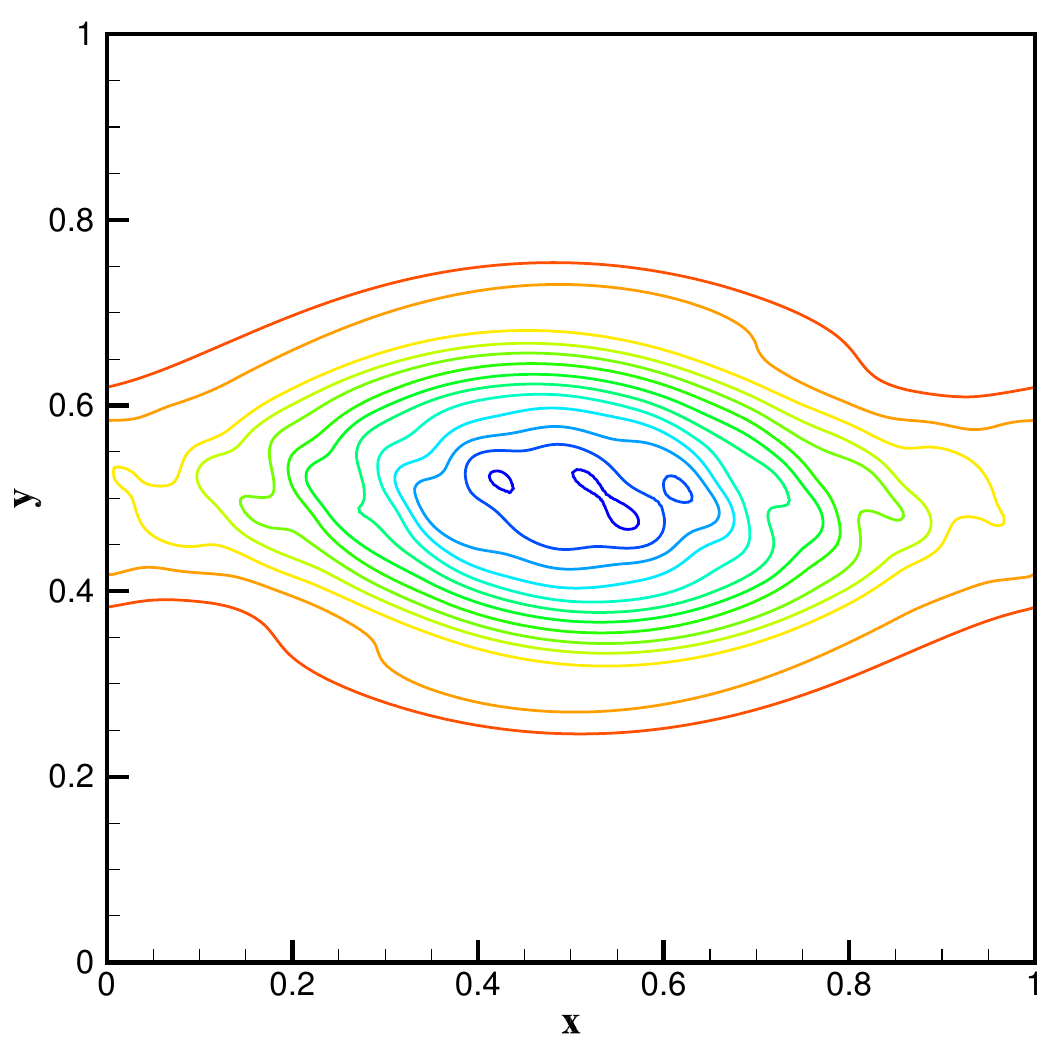}
	\caption{Evolution of the vorticity at different times $t=0.35,\ 0.6,\ 1.5,\  1.8,\ 5,\ 8$.
	}\label{KH}
\end{figure}
\subsection{Molecular beam epitaxial without slope selection}
In this section, we investigates the molecular beam epitaxial (MBE) model without slope selection \cite{cho1975molecular,schwoebel1966step,villain1991continuum}  in the domain $[0, 2\pi)^2$. The nonlinear functional associated with this model is not bounded from below. The classical no-slope MBE model is described by a fourth-order parabolic equation, which corresponds to the $L^2$ gradient flow of the Ehrlich-Schwoebel energy functional. This energy functional is defined as:
\begin{equation*}
	E(\phi)=\int_\Omega \frac{\epsilon^2}{2}|\Delta \phi|^2+F(\phi)d\textbf{x},
\end{equation*}
where $F(\phi)=-\frac{1}{2}\ln(1+|\nabla\phi|^2)$.
The evolution of the MBE model is governed by the equation:
\begin{equation*}
	\phi_t=-\mathcal{M}\frac{\delta E(\phi)}{\delta \phi}= -\mathcal{M}(\epsilon^2\Delta^2\phi+F'(\phi)),
\end{equation*}
Here, $\mathcal{M}$ is the mobility constant, and the nonlinear force vector is defined by 
$F'(\phi):=\nabla\cdot(\frac{\nabla\phi}{1+|\nabla\phi|^2})$. By taking the inner product with the test function $-\frac{1}{\mathcal{M}}\phi_t$, we derive
\begin{equation}
	\frac{d}{dt}E(\phi)=-\mathcal{K}(\phi).
\end{equation}
where $\mathcal{K}(\phi)=\frac{1}{\mathcal{M}}\|\phi_t\|^2$.
Similar to section \ref{NSeq}, we can readily develop unconditionally stable numerical schemes for the MBE equation without slope selection, utilizing the framework provided in \eqref{CN-SM3}. Specifically, by introducing a new auxiliary variable  $V(t)=\theta E(\phi)+C_0$, 
the scheme can be formulated as follows:
\begin{equation}\label{MBE_scheme}
	\begin{array}{l}
		\displaystyle\frac{\phi^{n+1}-\phi^{n}}{\Delta t}+\mathcal{M}\epsilon^2\Delta^2\frac{\phi^{n+1}+\phi^{n}}{2}+\mathcal{M}
		F'\left(\eta^{n+\frac12}\overline{\phi}^{n+\frac12}\right)=0,\\
		\displaystyle\frac{\arctan V^{n+
				\frac12}-\arctan V^{n-\frac12}}{\Delta t}=-\frac{\theta\mathcal{K}(\phi^n)}{\theta^2E_{tot}^2(\phi^n)+1},\\
		\displaystyle \eta^{n+\frac12}=\chi(V^{n+\frac12}).
	\end{array}
\end{equation}
\subsubsection{Convergence tests}
We begin by evaluating the convergence rates of the CN-SM scheme \eqref{MBE_scheme} using a specific example. Assume the exact solution is given by
\begin{equation}
	\phi(x,y,t)=\cos(x)\cos(y)exp(-t).
\end{equation}
The computational domain is defined as $[0, 2 \pi)^2$, and we discretize the spatial domain using a grid of $256 \times 256$ points. 
We first perform the mesh refinement test in time with parameters set to 
$\mathcal{M}=0.1$ and $\epsilon=0.1$ at $T= 1.0$.
We observe that all combinations of parameters lead to convergence rates that asymptotically match their expected orders in time in Fig. \ref{MBE_order}. Notably, the choice of $\theta=1$ gives better accuracy than $\theta=0.1$ or $0.01$ when using the same time step, although all three choices are of the same order of magnitude.
\begin{figure}[!t]
	\centering 
	\subfloat[Time step size vs. errors of $\phi$]{
		\begin{minipage}[c]{0.45\textwidth}
			\includegraphics[width=1\textwidth]{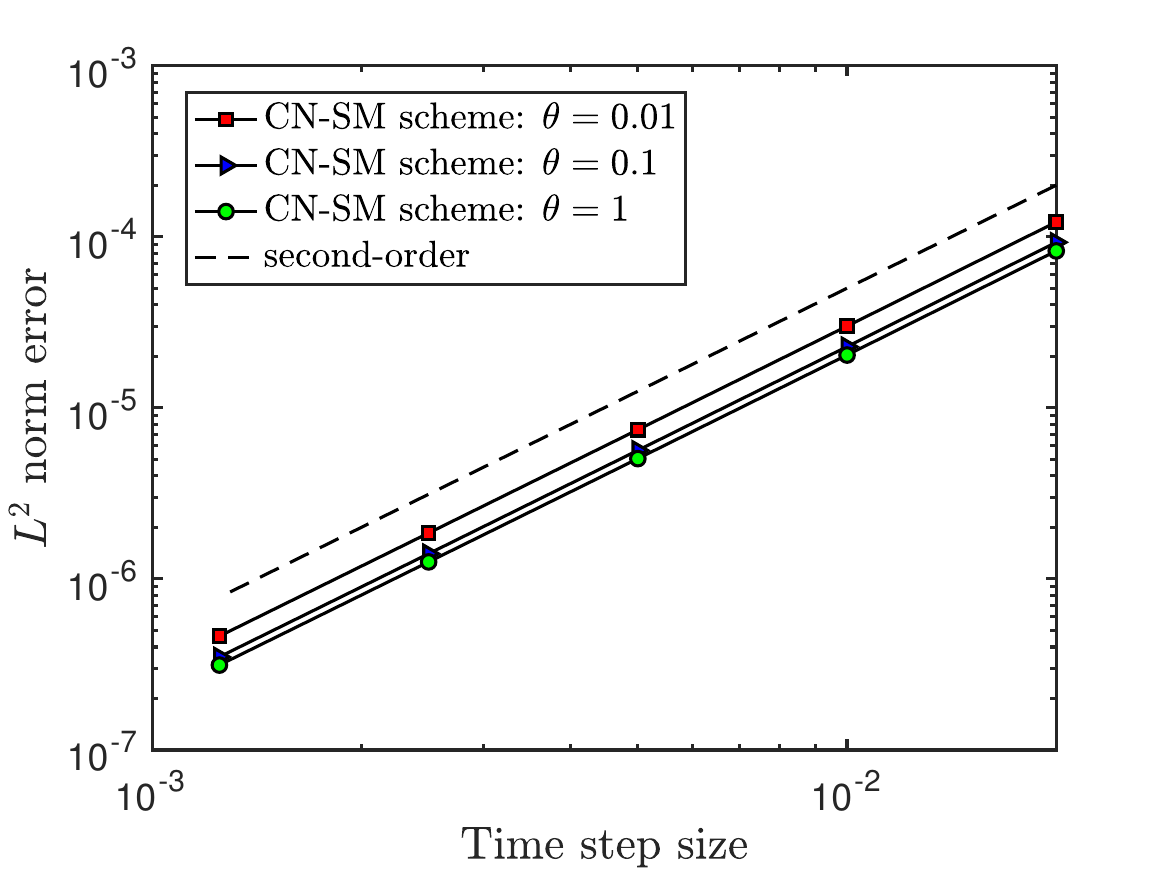}
		\end{minipage}
	}
	\subfloat[Time step size  vs. errors of $E$]{
		\begin{minipage}[c]{0.45\textwidth}
			\includegraphics[width=1\textwidth]{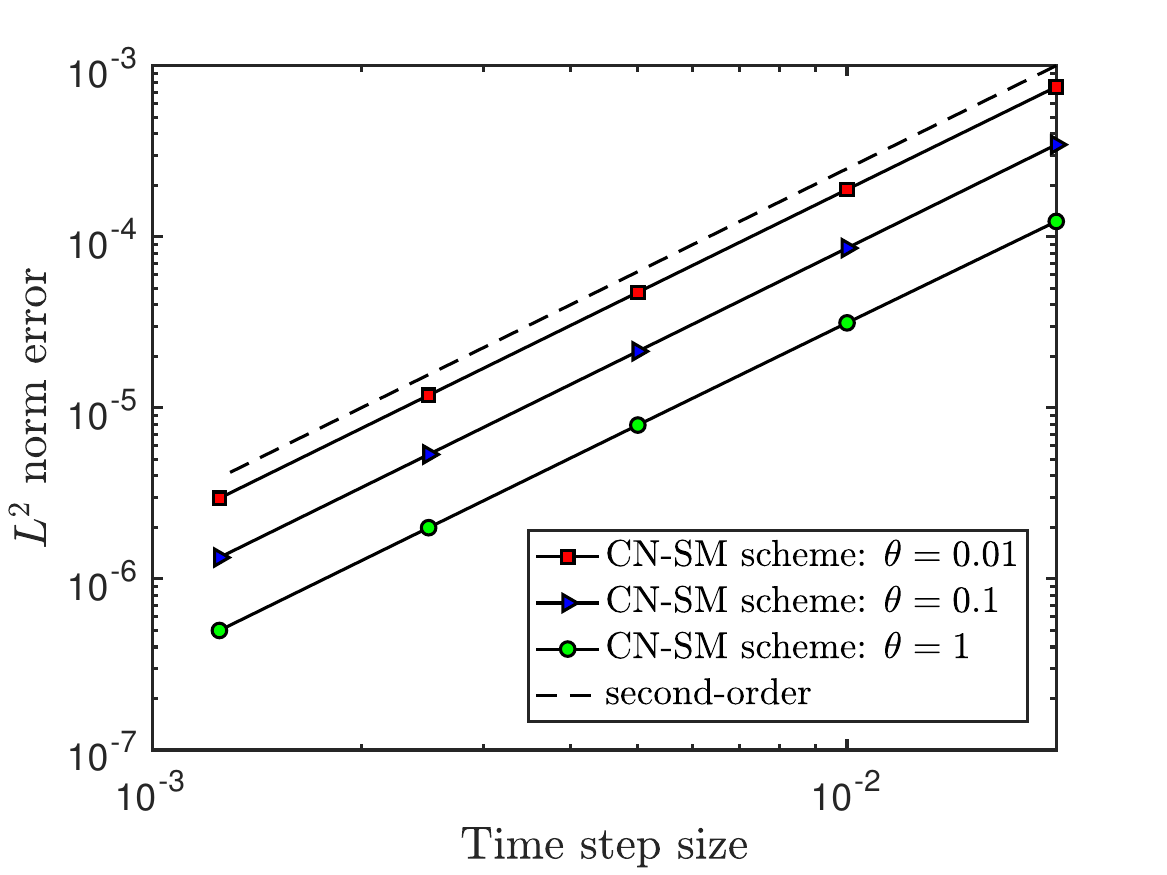}
		\end{minipage}
	}
	\caption{Numerical convergence rate of MBE model.
	}\label{MBE_order}
\end{figure}
\subsubsection{Coarsening dynamics}
In this example, we perform numerical simulations of coarsening dynamics in 2D.
The initial condition is specified as:
\begin{equation*}
	\phi(\boldsymbol{x},0)=0.001\text{rand}(\boldsymbol{x}),
\end{equation*}
where $\text{rand}(\boldsymbol{x})$ denotes random values within the interval $[-1, 1]^2$.
The simulations are carried out in the domain $[0,L)^2$ with $L=12.8$, and we discretize the spatial domain using a grid of $200 \times 200$ points. 
The parameters used  in the simulation are $\mathcal{M}=1$, $\epsilon=0.03$, $\theta=0.01$ and $C_0=10^5$ with a time step of $\Delta t=1\times10^{-3}$. 
To obtain the deviation of the height function, we define the roughness measure function $W(t)$ as:
$$W(t)=\sqrt{\frac{1}{|\Omega|} \int_{\Omega}(\phi(\boldsymbol{x}, t)-\bar{\phi}(\boldsymbol{x}, t))^2 d \boldsymbol{x}},$$
where 
$\bar{\phi}(t)=\frac{1}{\Omega} \int_{\Omega} \phi(\boldsymbol{x}, t) d \boldsymbol{x}$ is the spatial average at $t$.

We illustrate the isolines of the numerical solutions for the height function $\phi$ and its Laplacian $\Delta \phi$ for the model without slope selection in Fig. \ref{MBE}.
As depicted in Fig. \ref{MBE_EW}(a), the energy demonstrates a rapid decay that approximates the function  $-log_{10}(t)$,  which is consistent with the predictions made in \cite{wang2010unconditionally}. 
The growth rate of the roughness, shown in Fig. \ref{MBE_EW}(b), behaves as $t^{1/2}$.
Notably, these numerical solutions reveal features that closely align with those obtained in \cite{wang2010unconditionally}, where the convex splitting technique was used.

\begin{figure}[!t]
	\centering
	\subfloat[T=0]{
		\begin{minipage}[c]{0.45\textwidth}
			\includegraphics[width=1\textwidth]{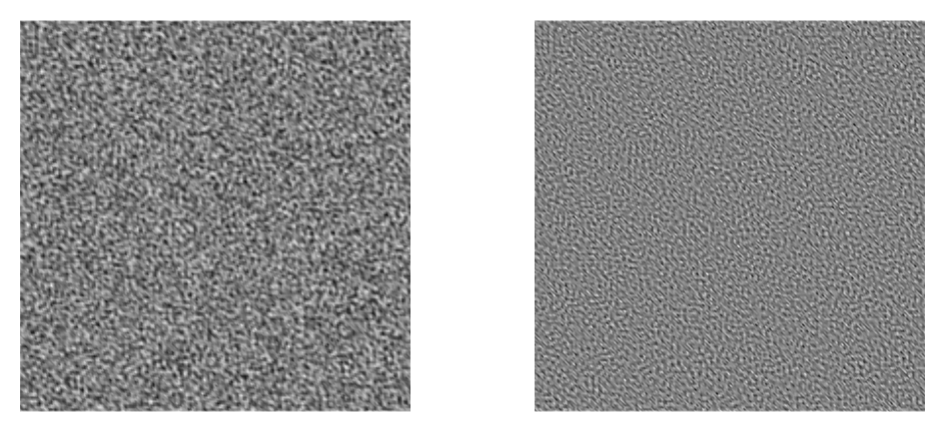}
		\end{minipage}
	}
	\subfloat[T=1]{
		\begin{minipage}[c]{0.45\textwidth}
			\includegraphics[width=1\textwidth]{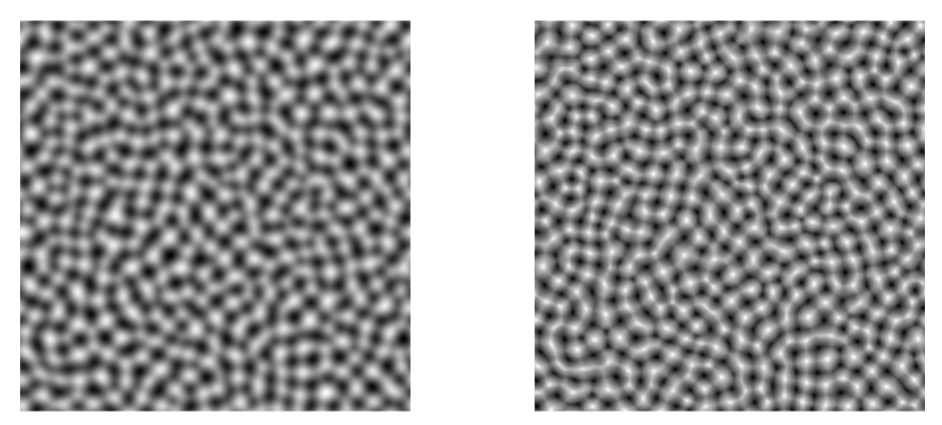}
		\end{minipage}
	}\\
	\subfloat[T=10]{
		\begin{minipage}[c]{0.45\textwidth}
			\includegraphics[width=1\textwidth]{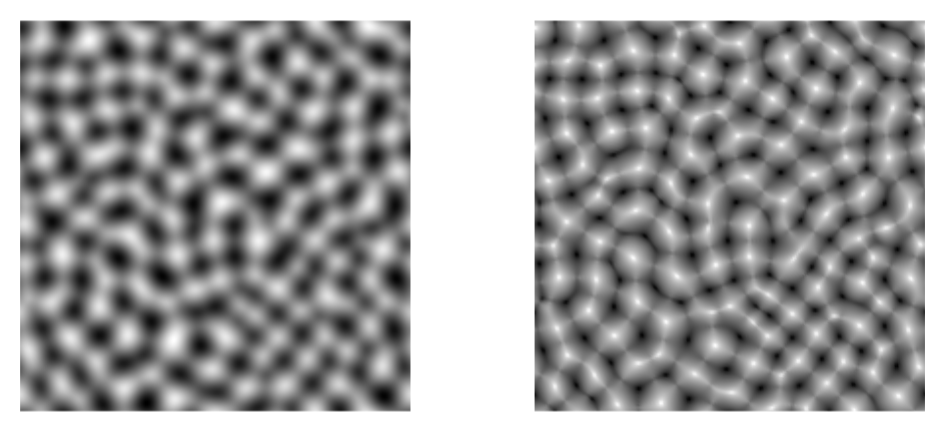}
		\end{minipage}
	}
	\subfloat[T=50]{
		\begin{minipage}[c]{0.45\textwidth}
			\includegraphics[width=1\textwidth]{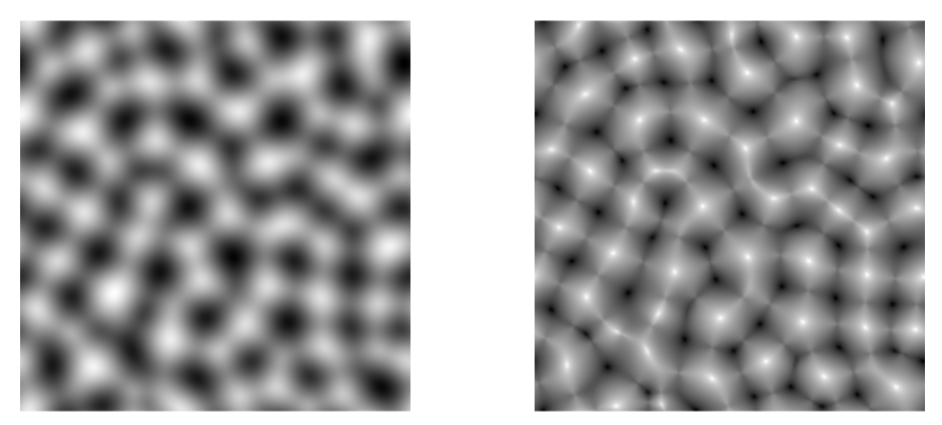}
		\end{minipage}
	}\\
	\subfloat[T=100]{
		\begin{minipage}[c]{0.45\textwidth}
			\includegraphics[width=1\textwidth]{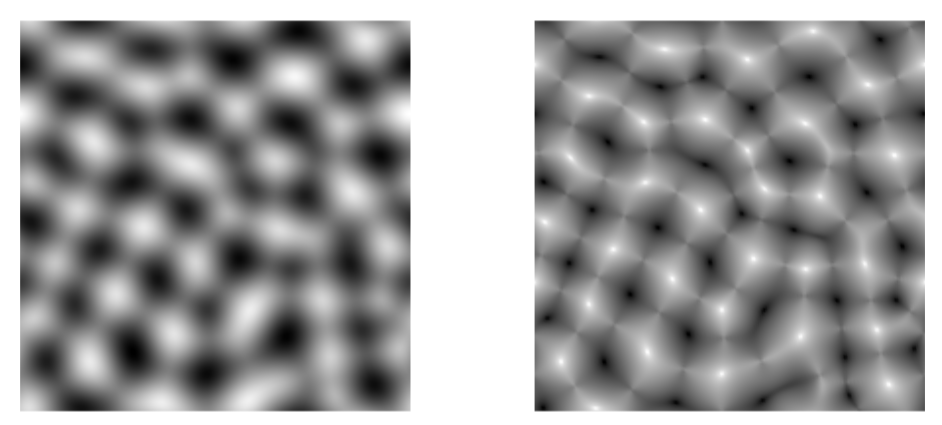}
		\end{minipage}
	}
	\subfloat[T=500]{
		\begin{minipage}[c]{0.45\textwidth}
			\includegraphics[width=1\textwidth]{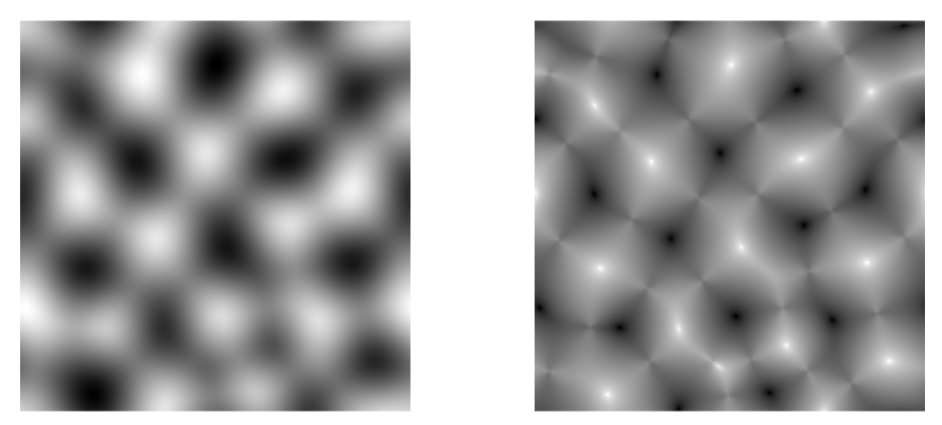}
		\end{minipage}
	}
	{\caption{The left subfigure represents $\phi$ and the right subfigure represents $\Delta \phi$. Each snapshot is taken at $T = 0, 1, 10, 50, 100, 500$.}\label{MBE}}
\end{figure}

\begin{figure}[!t]
	\centering 
	\subfloat[]{
		\begin{minipage}[c]{0.45\textwidth}
			\includegraphics[width=1\textwidth]{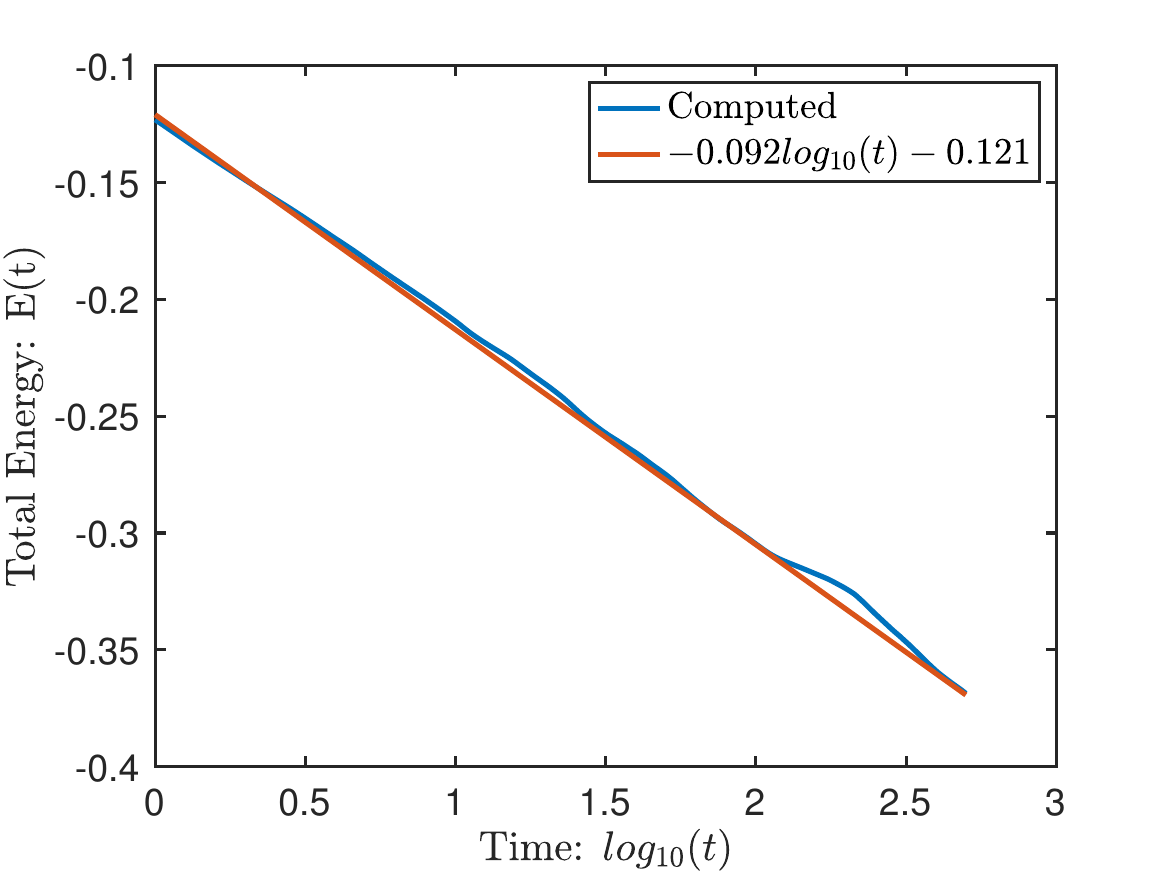}
		\end{minipage}
	}
	\subfloat[]{
		\begin{minipage}[c]{0.45\textwidth}
			\includegraphics[width=1\textwidth]{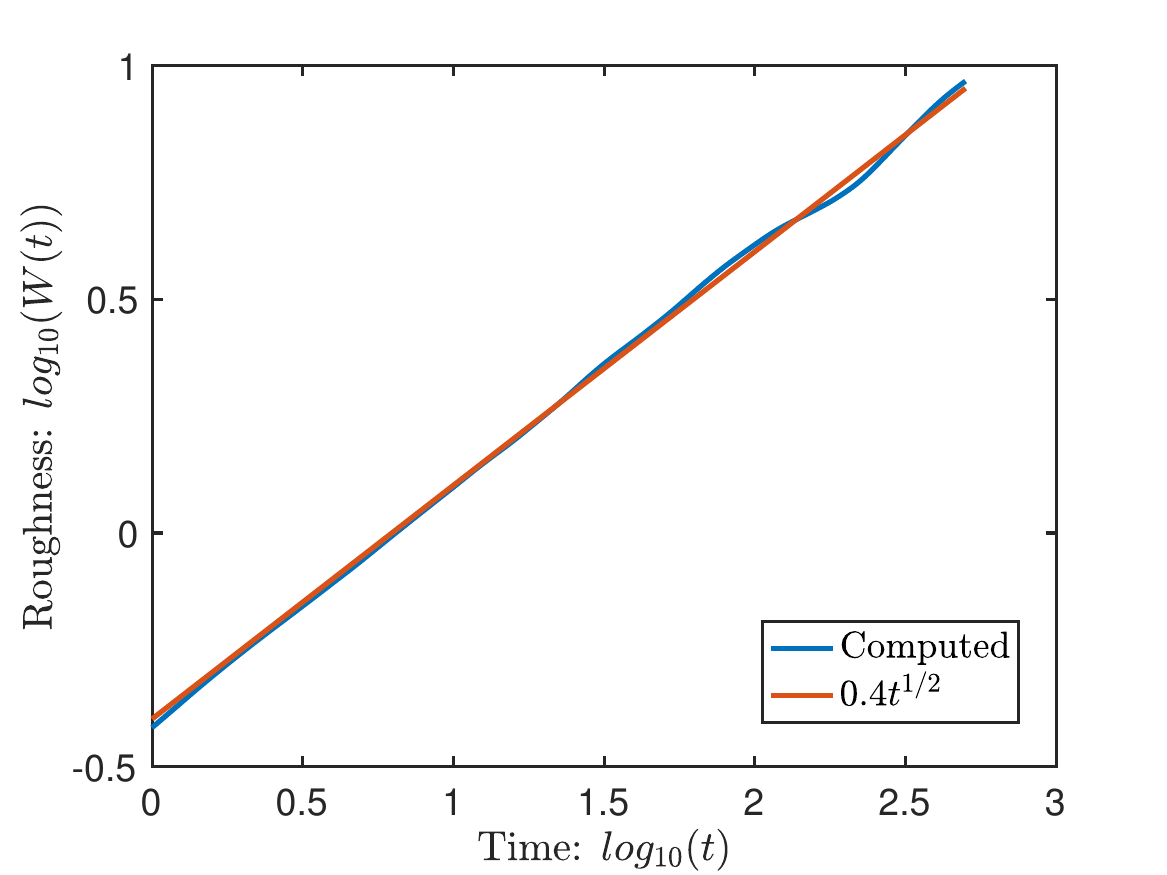}
		\end{minipage}
	}
	\caption{(a) Semi-log plot of energy, and (b) log-log plot of the standard deviation of $\phi$, denoted as $W(t)$.
		The energy decays at a rate of $-log_{10}(t)$, and $W(t)$ increases like $t^{1/2}$. 
	}\label{MBE_EW}
\end{figure}

\subsection{Ternary Cahn-Hilliard phase-field model}
This subsection presents the three-phase Cahn-Hilliard phase-field model, with the free energy for this system expressed as follows:
\begin{equation}\label{energy}
	\mathcal{E}\left[\phi_1, \phi_2, \phi_3\right]=\frac{3 \epsilon^2}{8} \sum_{l=1}^3 \int_{\Omega} \Sigma_l\left|\nabla \phi_l\right|^2 \textbf{dx}+12 \int_{\Omega} F\left(\phi_1, \phi_2, \phi_3\right) \textbf{dx}.
\end{equation}
The relationship among the phase fields is governed by the condition:
$$
\phi_1+\phi_2+\phi_3=1,
$$
The numerical approximation of the three-phase Cahn-Hilliard system is challenging due to the nonlinear terms present in the model, as outlined by Boyer et al \cite{boyer2006study,boyer2011numerical}. To address this issue, we will apply the Staggered mesh approach, as detailed in Section \ref{SMmethod}.

For consistency with the two-phase system, the surface tension parameters $\sigma_{12}, \sigma_{13}$ and $\sigma_{23}$ must satisfy the following equations:
$$
\Sigma_1=\sigma_{12}+\sigma_{13}-\sigma_{23}, \quad \Sigma_2=\sigma_{12}+\sigma_{23}-\sigma_{13}, \quad \Sigma_3=\sigma_{13}+\sigma_{23}-\sigma_{12} .
$$
In addition, the volume conservation constraint, expressed as $\phi_3=1-\phi_1-\phi_2$, allows us to reformulate the energy functional \eqref{energy} as follows:
$$
\mathcal{E}\left[\phi_1, \phi_2\right]= \int_{\Omega} \frac{3 \epsilon^2}{8} \left(\Sigma_1\left|\nabla \phi_1\right|^2+\Sigma_2\left|\nabla \phi_2\right|^2+\Sigma_3\left|\nabla \phi_1+\nabla \phi_2\right|^2\right) +12 F\left(\phi_1, \phi_2\right) \textbf{dx},
$$
where $F\left(\phi_1, \phi_2\right)$ is defined as:
\begin{align*}
F\left(\phi_1, \phi_2\right)=&\frac{\Sigma_1}{2} \phi_1^2\left(1-\phi_1\right)^2+\frac{\Sigma_2}{2} \phi_2^2\left(1-\phi_2\right)^2+\frac{\Sigma_3}{2}\left(\phi_1+\phi_2\right)^2\left(1-\phi_1-\phi_2\right)^2\\
&+3 \Lambda \phi_1^2 \phi_2^2\left(1-\phi_1-\phi_2\right)^2,
\end{align*}
where $\Lambda$ is a non-negative constant.

We investigate the coupled Cahn-Hilliard equations, which can be formulated as follows:
\begin{subequations}\label{cch}
	\begin{align}
		& \partial_t \phi_l=\mathcal{M} \Delta \frac{\mu_l}{\Sigma_l}, \quad l=1,2, \label{cch11}
		\\
		& \mu_1=-\frac{3 \epsilon^2}{4}\left(\Sigma_1+\Sigma_3\right) \Delta \phi_1-\frac{3 \epsilon^2}{4} \Sigma_3 \Delta \phi_2+12 \frac{\partial F\left(\phi_1, \phi_2\right)}{\partial \phi_1}, \label{cch22}\\
		& \mu_2=-\frac{3 \epsilon^2}{4} \Sigma_3 \Delta \phi_1-\frac{3 \epsilon^2}{4}\left(\Sigma_2+\Sigma_3\right) \Delta \phi_2+12 \frac{\partial F\left(\phi_1, \phi_2\right)}{\partial \phi_2}.\label{cch33}
	\end{align}
\end{subequations}
Initial conditions are specified as follows:
$$
\left.\phi_l(\textbf{x}, t)\right|_{t=0}=\phi_l^0(\textbf{x}), \quad l=1,2, \quad \phi_3^0=1-\phi_1^0(\textbf{x})-\phi_2^0(\textbf{x}) .
$$
By taking the inner product of  \eqref{cch11} with $\mu_1$, and $\mu_2$, and of \eqref{cch22} with $-\partial_t \phi_1$,  as well as  \eqref{cch33} with $-\partial_t \phi_2$, results in the energy dissipation law:
$$
\frac{\mathrm{d}}{\mathrm{d} t}E(\phi_1,\phi_2)=
-\mathcal{K}(\phi_1,\phi_2)\leq 0 .
$$
where $$\mathcal{K}(\phi_1,\phi_2)=\mathcal{M}\left(\frac{1}{\Sigma_1}\left\|\nabla \mu_1\right\|^2+\frac{1}{\Sigma_2}\left\|\nabla \mu_2\right\|^2\right).$$
We define an auxiliary variable $V(t)=E(\phi_1,\phi_2)+C_0$, where $C_0 \geq 0$ ensures  that $V(t)>0$. 
Correspondingly, efficient and unconditionally stable numerical schemes can be implemented as follows:
\begin{equation}\label{cch1}
	\begin{aligned}
		& \frac{\phi_l^{n+1}-\phi_l^n}{\Delta t}= \frac{M}{\Sigma_l}\Delta \mu_l^{n+\frac{1}{2}}, \quad l=1,2,\\
		&	\mu_1^{n+\frac{1}{2}}=-\frac{3 \epsilon^2}{4}\left(\Sigma_1+\Sigma_3\right) \Delta \phi_1^{n+\frac{1}{2}}-\frac{3 \epsilon^2}{4} \Sigma_3 \Delta \phi_2^{n+\frac{1}{2}}+12\eta^{n+\frac12}\frac{\partial F}{\partial \phi_1}\left(\bar{\phi}_1^{n+\frac12}, \bar{\phi}_2^{n+\frac12}\right), \\
		&\mu_2^{n+\frac{1}{2}}=-\frac{3 \epsilon^2}{4} \Sigma_3 \Delta \phi_1^{n+\frac{1}{2}}
		-\frac{3 \epsilon^2}{4}(\Sigma_2+\Sigma_3) \Delta \phi_2^{n+\frac{1}{2}}+12\eta^{n+\frac12} \frac{\partial F}{\partial \phi_2}\left(\bar{\phi}_1^{n+\frac12}, \bar{\phi}_2^{n+\frac12}\right),\\
		&\displaystyle\frac{\ln V^{n+
				\frac12}-\ln V^{n-\frac12}}{\Delta t}=-\frac{\mathcal{K}(\phi^n_1,\phi^n_2)}{E(\phi^n_1,\phi^n_2)},\\
		&\displaystyle\eta^{n+\frac12}=\chi(V^{n+\frac12}).
	\end{aligned}
\end{equation}

\subsubsection{Accuracy test} 
This section illustrates the accuracy of the ternary Cahn-Hilliard system using the following initial conditions:
$$
\begin{aligned}
	& \phi_i^0(x, y)=\frac{1}{2}\left(1+\tanh \left(\frac{r_i-\sqrt{\left(x-x_i\right)^2+\left(y-y_i\right)^2}}{\epsilon}\right)\right), \quad i=1,2, \\
	& \phi_3^0(x, y)=1-\phi_1^0(x, y)-\phi_2^0(x, y),
\end{aligned}
$$
where $r_1=r_2=0.35,\ x_1=1.37,\ x_2=0.63$ and $y_1=y_2=1.0$. The computational domain is specified as $\Omega=[0,2]^2$,
and spatial variables are discretized on a $256\times256$ grid.
The parameters utilized in the simulation include $\mathcal{M}=10^{-5},\ \epsilon=0.02$, and $\Lambda=7$.
The corresponding $L^2$ errors calculated from \eqref{cch1} at $T=0.1$ with various surface tension strengths are 
presented in Fig. \ref{TCH}(a), with a reference solution derived from a time step of $\Delta t=1.0 \times 10^{-5}$. 

\begin{figure}[htp]
	\centering
	\subfloat[]{
		\begin{minipage}[c]{0.45\textwidth}
			\includegraphics[width=1\textwidth]{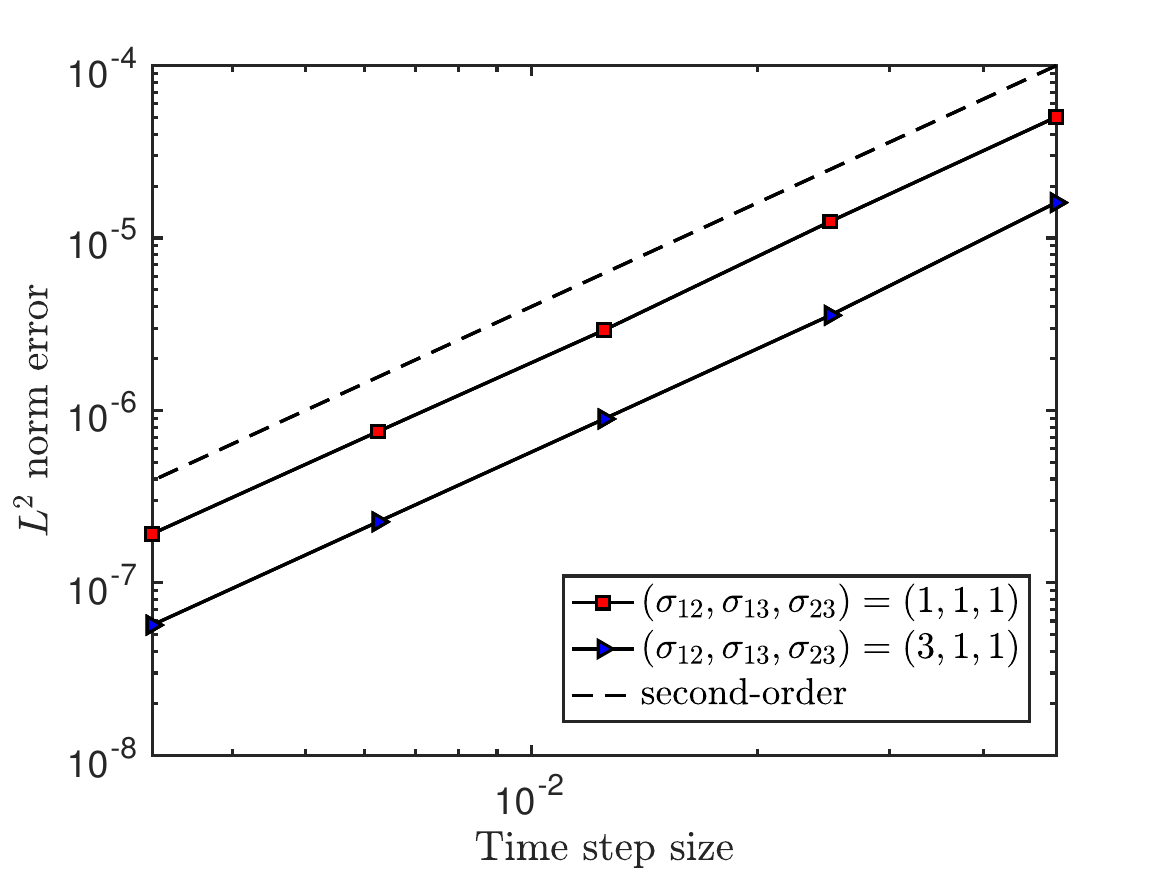}
		\end{minipage}
	}
	\subfloat[]{
		\begin{minipage}[c]{0.45\textwidth}
			\includegraphics[width=1\textwidth]{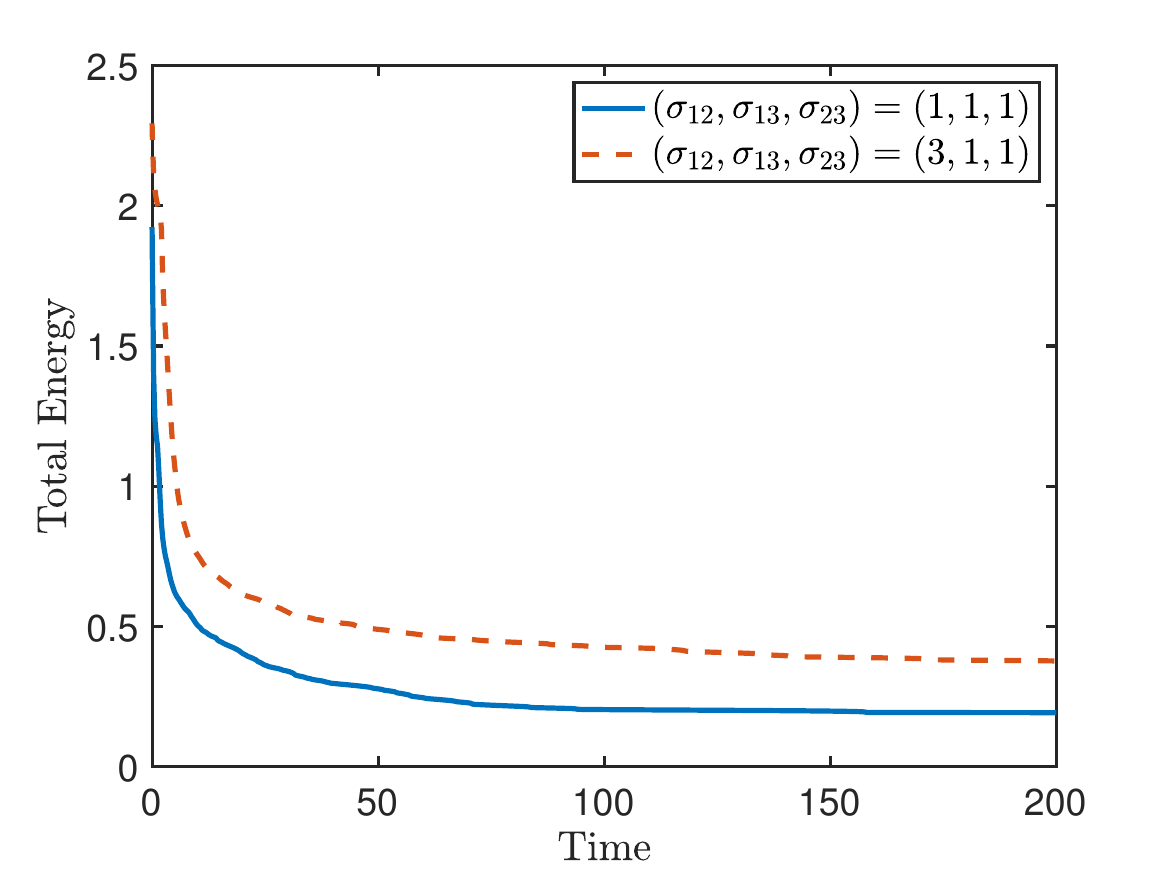}
		\end{minipage}
	}
	{\caption{(a) Numerical convergence rate  for the ternary Cahn-Hilliard phase-field model. (b)Temporal evolution of energy for the ternary Cahn-Hilliard phase-field model.}\label{TCH}}
\end{figure}

\subsubsection{Spinodal decomposition}
We study spinodal decomposition by observing the transformation of a uniform three-phase mixture into distinct phases as concentration fluctuations grow.
The initial conditions are specified as
$$
\begin{aligned}
	& \phi_1(x, y, 0)=0.5\left(\frac{y}{2}+0.25\right)+0.001 \operatorname{rand}(x, y), \\
	& \phi_2(x, y, 0)=0.5\left(\frac{y}{2}+0.25\right)+0.001 \operatorname{rand}(x, y), \\
	& \phi_3(x, y, 0)=1-\phi_1(x, y, 0)-\phi_2(x, y, 0),
\end{aligned}
$$
where $\text{rand}(x,y)$ denotes random values within the range $[-1, 1]^2$. 
The computational domain is defined as 
$\Omega=[0,2] \times[0,1]$, and spatial variables are discretized using $256\times 128$ grid points. The parameters are set to $\mathcal{M}=10^{-3}$, $\epsilon=0.025$,  and $\Lambda=7$. Fig. \ref{TCH}(b) illustrates the energy evolution over time in the ternary Cahn-Hilliard phase-field model, using a time step of $\Delta t=2.0\times 10^{-4}$. In addition, we test various surface tension strengths, such as  $\left(\sigma_{12}, \sigma_{13}, \sigma_{23}\right)=(1,1,1)$ and $\left(\sigma_{12}, \sigma_{13}, \sigma_{23}\right)=(3,1,1)$, as shown in Fig. \ref{TCH1} and Fig. \ref{TCH2}, which are consistent with \cite{hong2023high}.

\begin{figure}[htp]
	\centering
	\subfloat[T=10]{
		\begin{minipage}[c]{0.3\textwidth}
			\includegraphics[width=1\textwidth]{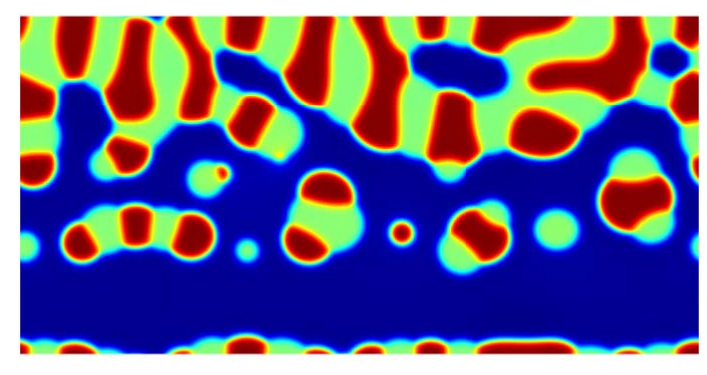}
		\end{minipage}
	}
	\subfloat[T=20]{
		\begin{minipage}[c]{0.3\textwidth}
			\includegraphics[width=1\textwidth]{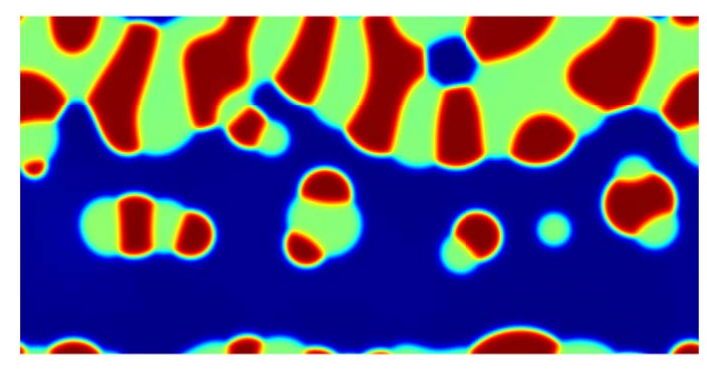}
		\end{minipage}
	}
	\subfloat[T=50]{
		\begin{minipage}[c]{0.3\textwidth}
			\includegraphics[width=1\textwidth]{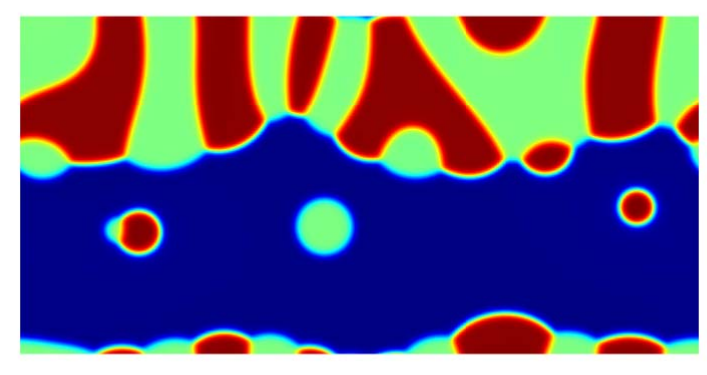}
		\end{minipage}
	}\\
	\subfloat[T=100]{
		\begin{minipage}[c]{0.3\textwidth}
			\includegraphics[width=1\textwidth]{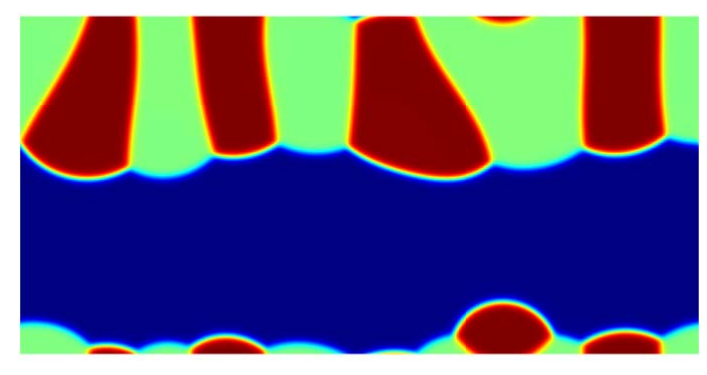}
		\end{minipage}
	}
	\subfloat[T=150]{
		\begin{minipage}[c]{0.3\textwidth}
			\includegraphics[width=1\textwidth]{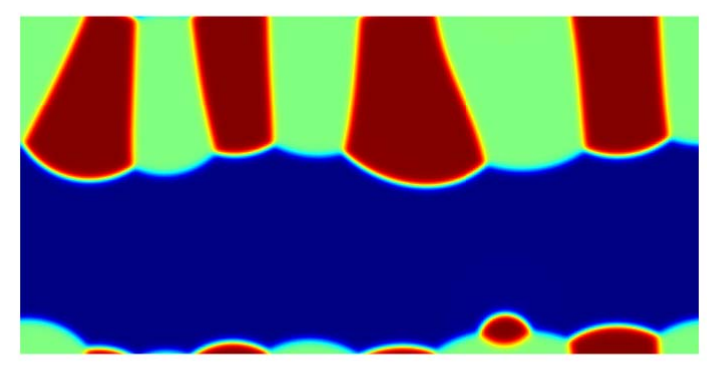}
		\end{minipage}
	}
	\subfloat[T=200]{
		\begin{minipage}[c]{0.3\textwidth}
			\includegraphics[width=1\textwidth]{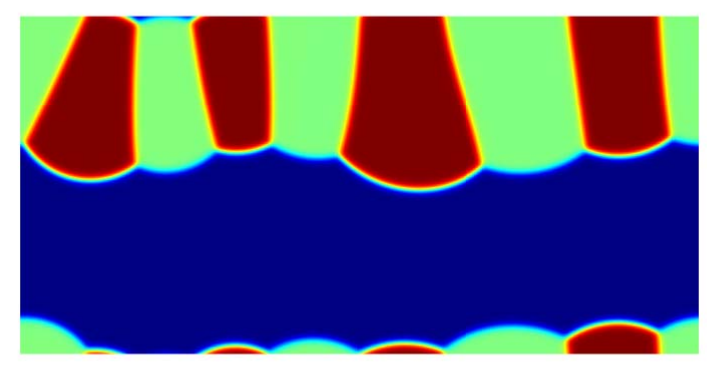}
		\end{minipage}
	}
	{\caption{Dynamical evolutions of the profile $\frac{1}{2}\phi_1+\phi_2$ for the spinodal decomposition examples with $\left(\sigma_{12}, \sigma_{13}, \sigma_{23}\right)=(1,1,1)$. }\label{TCH1}}
\end{figure}

\begin{figure}[htp]
	\centering
	\subfloat[T=10]{
		\begin{minipage}[c]{0.3\textwidth}
			\includegraphics[width=1\textwidth]{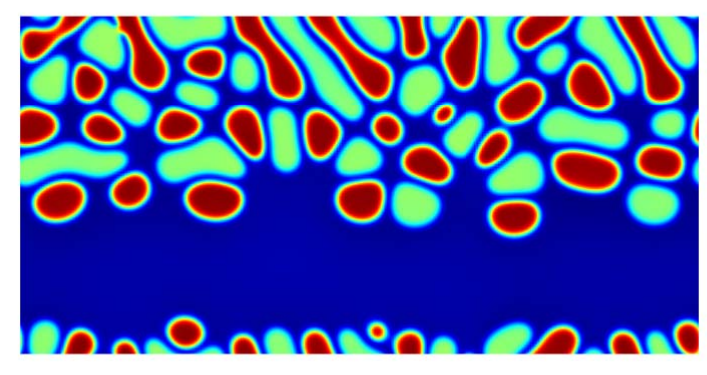}
		\end{minipage}
	}
	\subfloat[T=20]{
		\begin{minipage}[c]{0.3\textwidth}
			\includegraphics[width=1\textwidth]{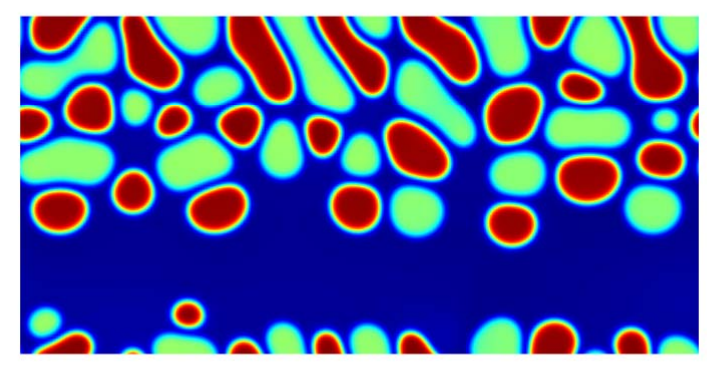}
		\end{minipage}
	}
	\subfloat[T=50]{
		\begin{minipage}[c]{0.3\textwidth}
			\includegraphics[width=1\textwidth]{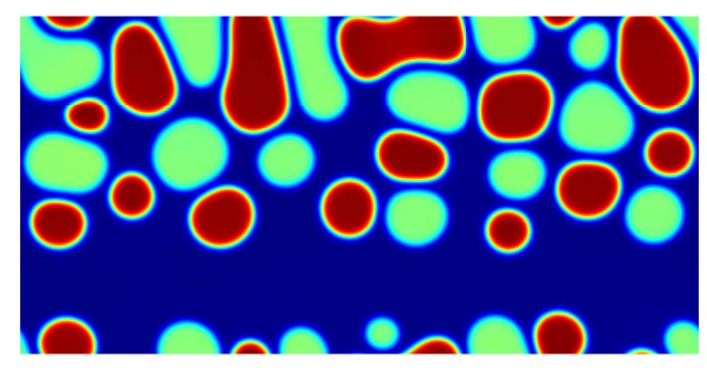}
		\end{minipage}
	}\\
	\subfloat[T=100]{
		\begin{minipage}[c]{0.3\textwidth}
			\includegraphics[width=1\textwidth]{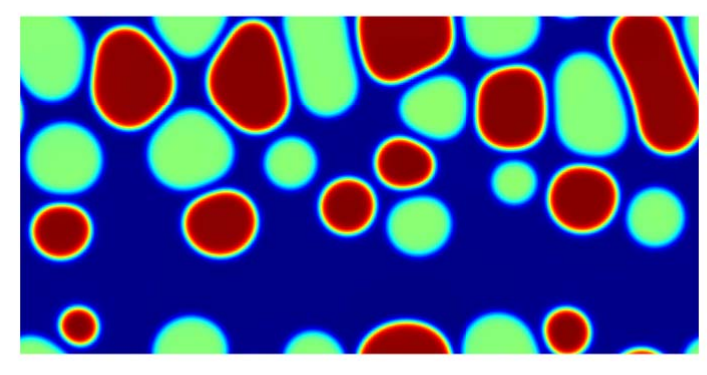}
		\end{minipage}
	}
	\subfloat[T=150]{
		\begin{minipage}[c]{0.3\textwidth}
			\includegraphics[width=1\textwidth]{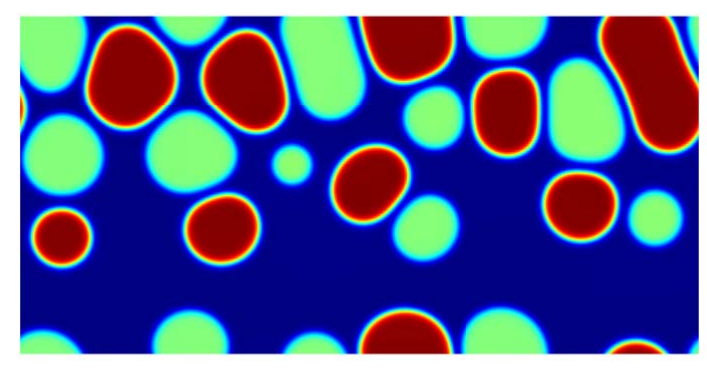}
		\end{minipage}
	}
	\subfloat[T=200]{
		\begin{minipage}[c]{0.3\textwidth}
			\includegraphics[width=1\textwidth]{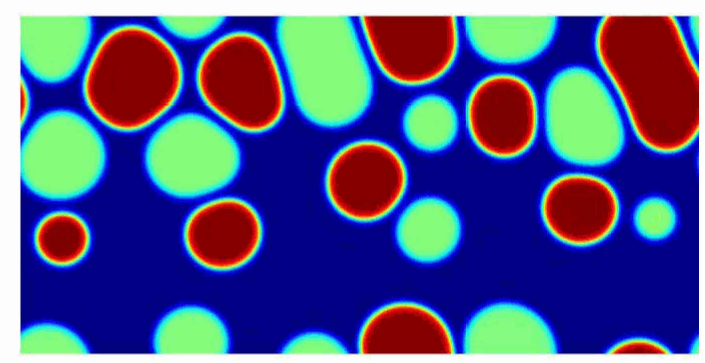}
		\end{minipage}
	}
	{\caption{Dynamical evolutions of the profile $\frac{1}{2}\phi_1+\phi_2$ for the spinodal decomposition examples with $\left(\sigma_{12}, \sigma_{13}, \sigma_{23}\right)=(3,1,1)$. }\label{TCH2}}
\end{figure}

\section{Concluding remarks} 
The recently SAV-type approach mostly needs to assume that the total energy or the nonlinear energy functional has a clear lower bound. For the classic dissipative systems such as Allen-Cahn and Cahn-Hilliard model, we know that their total energy and nonlinear energy functional are all non-negative for a double-well potential. This means that it is not necessary to choose a very large constant $C$ to ensure the positivity of the auxiliary variables and the discrete energy when performing numerical calculations. However, for some dissipative systems such as phase field crystal model and molecular beam epitaxial model, the energy lower bounds are not clear before calculation. We have to introduce a relatively big constant $C$ to ensure $E_{tot}(u)+C>0$. Once $C$ is not chosen large enough, the results of the SAV-based algorithms may not be accurate enough.   

In this paper, we propose several staggered mesh schemes for general nonlinear dissipative system with known and unknown energy lower bounds. By introducing different auxiliary variables and discretize them on staggered meshes in time, we construct several linear, decoupled and unconditional energy stable schemes no matter the energy has clear lower bounds or not for dissipative systems. The new proposed schemes can not only preserve the positivity of the introduced auxiliary variables, but also be able to get a second-order approximation to the energy. Besides, we give several proper choices for the introduced auxiliary variables $V(t)$ and $\eta(t)$ to simulate various complex dissipative systems. In numerical examples, by a comparison study with other numerical methods such as Lagrange Multiplier method and GSAV method, it fully shows the advantages of our proposed schemes in accuracy and computational efficiency. It is worth mentioning that the new SM schemes also show excellent results when simulating some complex 
benchmark problems for the Navier-Stokes equations.    	
\section*{Acknowledgement}
No potential conflict of interest was reported by the author. We would like to acknowledge the assistance of volunteers in putting together this example manuscript and supplement.

\bibliographystyle{siamplain}
\bibliography{SM}

\end{document}